\newcommand\CC{{\mathbb C}}

\newcommand\cC{{\cal C}}
\newcommand\cD{{\cal D}}

\newcommand\cG{{\cal G}}
\newcommand\cH{{\cal H}}
\newcommand\cI{{\cal I}}                               

\newcommand\cK{{\cal K}}
\newcommand\cL{{\cal L}}
\newcommand\cM{{\cal M}}
\newcommand\cN{{\cal N}}

\newcommand\cO{{\cal O}}              
\newcommand\cP{{\cal P}}  
\newcommand\cQ{{\cal Q}}
\newcommand\cR{{\cal R}}
\newcommand\cS{{\cal S}}

\newcommand\cU{{\cal U}}
\newcommand\cV{{\cal V}} 

\newcommand\cX{{\cal X}} 

\newcommand\cZ{{\cal Z}}

\newcommand\DD{{\mathbb D}}
\newcommand\dra{\dashrightarrow}
\newcommand\es{\emptyset}

\newcommand\GG{{\mathbb G}}

\newcommand\gM{\mathfrak{M}}
\newcommand\gN{\mathfrak{N}}

\newcommand\gp{\mathfrak{p}}

\newcommand\gT{\mathfrak{T}}

\newcommand\hra{\hookrightarrow}
\newcommand\la{\langle}

\newcommand\lagr{\mathbb{LG}(\wedge^3 V)}

\newcommand\lra{\longrightarrow}
\newcommand\n{\noindent}

\newcommand\ov{\overline}
\newcommand\PP{{\mathbb P}}
\newcommand\QQ{{\mathbb Q}}
\newcommand\ra{\rangle}
\newcommand\RR{{\mathbb R}}

\newcommand\ul{\underline}

\newcommand\wh{\widehat}
\newcommand\wt{\widetilde}
\newcommand\ZZ{{\mathbb Z}}
\newcommand{\mapor}[1]{{\stackrel{#1}{\longrightarrow}}}
\newcommand{\mapver}[1]{\Big\downarrow
\vcenter{\rlap{$\scriptstyle#1$}}}
\documentclass[10pt,a4paper]{article}
\usepackage{amsthm}
\usepackage{amsmath}
\usepackage{amssymb}
\usepackage{makeidx}         
\usepackage{graphicx}        
\usepackage{multicol}        
\usepackage[all,ps]{xy}
\DeclareMathOperator{\alb}{alb}

\DeclareMathOperator{\cod}{cod}

\DeclareMathOperator{\coker}{coker}

\DeclareMathOperator{\Pic}{Pic}

\DeclareMathOperator{\Spec}{Spec}

\DeclareMathOperator{\vol}{vol}

\usepackage{ifthen}
\newcommand{\cit}[1]{{\rm \textbf{#1}}}
\newcommand{\Ref}[2]{\cit{%
\ifthenelse{\equal{#1}{thm}}{Theorem}{}%
\ifthenelse{\equal{#1}{ass}}{Assumption}{}%
\ifthenelse{\equal{#1}{prp}}{Proposition}{}%
\ifthenelse{\equal{#1}{lmm}}{Lemma}{}%
\ifthenelse{\equal{#1}{crl}}{Corollary}{}%
\ifthenelse{\equal{#1}{cnj}}{Conjecture}{}%
\ifthenelse{\equal{#1}{dfn}}{Definition}{}%
\ifthenelse{\equal{#1}{expl}}{Example}{}%
\ifthenelse{\equal{#1}{hyp}}{Hypothesis}{}%
\ifthenelse{\equal{#1}{rmk}}{Remark}{}%
\ifthenelse{\equal{#1}{clm}}{Claim}{}%
\ifthenelse{\equal{#1}{exe}}{Exercise}{}%
\ifthenelse{\equal{#1}{qst}}{Question}{}%
\ifthenelse{\equal{#1}{sec}}{Section}{}%
\ifthenelse{\equal{#1}{subsec}}{Subsection}{}%
\ifthenelse{\equal{#1}{univ}}{Universal Property}{}%
\ifthenelse{\equal{#1}{trm}}{Terminology}{}%
\  \ref{#1:#2}%
}}

\theoremstyle{plain}

\newtheorem{thm}{Theorem}[section]

\newtheorem{clm}[thm]{Claim}

\newtheorem{crl}[thm]{Corollary}

\newtheorem{lmm}[thm]{Lemma}
\newtheorem{prp}[thm]{Proposition}
\newtheorem{prp-dfn}[thm]{Proposition-Definition}

\theoremstyle{definition}

\newtheorem{ass}[thm]{Assumption}
\newtheorem{dfn}[thm]{Definition}

\theoremstyle{remark}

\newtheorem{cnj}[thm]{Conjecture}

\newtheorem{qst}[thm]{Question}
\newtheorem{rmk}[thm]{Remark}

 \makeindex
 
\begin{document}
 \title{Higher-dimensional analogues of $K3$ surfaces}
 \author{Kieran G. O'Grady\thanks{Supported by
 Cofinanziamento M.U.R. 2008-2009}\\\\
\lq\lq Sapienza\rq\rq Universit\`a di Roma}
\date{May 17 2010}
 \maketitle
  \tableofcontents
 \section{Introduction}\label{prologo}
 \setcounter{equation}{0}
  $K3$ surfaces were known classically as complex smooth projective surfaces whose generic hyperplane section is a  canonically embedded curve;  an example is provided by a smooth quartic surface in $\PP^3$. One naturally encounters $K3$'s in the Enriques-Kodaira classification of compact complex surfaces: they are defined to be compact K\"ahler surfaces with trivial canonical bundle and vanishing first Betti number. Below we list a few among the wonderful properties  of these surfaces:
\begin{itemize}
\item[(1)]
(Kodaira~\cite{kod}): Any two $K3$ surfaces are deformation equivalent - thus they are all deformations of a quartic surface.
\item[(2)]
The K\"ahler cone of a $K3$ surface $X$ is described as follows. Let $\omega\in H^{1,1}_{\RR}(X)$ be one K\"ahler class and  $\cN_X$ be the set of nodal classes 
\begin{equation}
\cN_X :=\{\alpha\in H^{1,1}_{\ZZ}(X)\mid \alpha\cdot\alpha=-2,\quad
\alpha\cdot\omega >0 \}
\end{equation}
The K\"ahler cone $\cK_X$ is given by
\begin{equation}
\cK_X :=\{\alpha\in H^{1,1}_{\RR}(X)\mid  \alpha\cdot\alpha>0,\quad
\alpha\cdot\beta >0\quad \forall \beta\in\cN_X \}.
\end{equation}
\item[(3)]
(Shafarevich \& Piatechki - Shapiro~\cite{piatsha}, Burns \& Rapoport~\cite{bura}, Looijenga \& Peters~\cite{looi}): Weak and strong Global Torelli hold. The weak version states that two $K3$ surfaces $X,Y$ are isomorphic if and only if there exists an integral isomorphism of Hodge structures $f\colon H^2(X)\overset{\sim}{\lra} H^2(Y)$ which is an isometry (with respect to  the intersection forms), the strong version states that $f$ is induced by an isomorphism $\phi\colon Y\overset{\sim}{\lra} X$ if and only if it maps effective divisors to effective divisors\footnote{Effective divisors have a purely Hodge-theoretic description once we have located one K\"ahler class.}. 
\end{itemize}
 The higher-dimensional  compact  K\"ahler  manifolds closest to $K3$ surfaces are {\it hyperk\"ahler manifolds} (HK); they are defined to be  simply connected   with $H^{2,0}$ spanned by the class of a holomorphic \ul{symplectic} form. The terminology originates from riemannian geometry: Yau's solution of Calabi's conjecture  gives that every  K\"ahler class $\omega$ on a HK manifold contains a K\"ahler metric $g$ with holonomy the compact symplectic group. There is a sphere $S^2$ (the  pure  quaternions of norm $1$) parametrizing complex structures for which $g$ is a K\"ahler metric - the {\it twistor family} associated to $g$; it plays a key role in the general theory of HK manifolds\footnote{Hyperk\"ahler manifolds are also known as {\it irreducible symplectic}}.
 Notice that a HK manifold  has trivial canonical bundle and is of even dimension. An example of Beauville~\cite{beau}  of dimension $2n$: the Douady space $S^{[n]}$ parametrizing length-$n$ analytic subsets of a $K3$ surface $S$. (Of course $S^{[n]}$ is a Hilbert scheme if $S$ is projective.)
  We mention right away two results which suggest that HK manifolds might behave like $K3$'s. Let $X$ be HK: 
 \begin{itemize}
\item[(a)]
By a theorem of Bogomolov~\cite{bog1} deformations of $X$ are unobstructed\footnote{The obstruction space $H^2(T_X)$ might be non-zero, e.g.~if $X$ is a generalized Kummer, see Subsection~\ref{esembeau}.} i.e.~the deformation space $Def(X)$ is smooth of the expected dimension $H^1(T_X)$.
\item[(b)]
Since the sheaf map $T_X\to\Omega^1_X$ given by contraction with a holomorphic symplectic form is an isomorphism it follows that the differential of the weight-$2$ period map 
 \begin{equation}\label{derper}
H^1(T_X)\lra \text{Hom}(H^{2,0}(X), H^{1,1}(X)) 
\end{equation}
is injective i.e.~infinitesimal Torelli holds.
\end{itemize}
We notice that by~(a) the generic deformation of $X$ has $h^{1,1}_{\ZZ}=0$ - in particular it is not projective. In fact given  $\alpha\in H^{1}(\Omega^1_X)$ and a first order deformation $\kappa\in H^1(T_X)$ we know by Griffiths that  $\alpha$ remains of type $(1,1)$ to first order in the direction $\kappa$ if and only if $Tr(\kappa\cup\alpha) =0$, moreover  the map
 \begin{equation}\label{ostia}
\begin{matrix}
H^1(T_X) & \lra & H^{2}(\cO_X) \\
\kappa & \mapsto & Tr(\kappa\cup\alpha) 
\end{matrix}
\end{equation}
is surjective if $\alpha\not=0$ by Serre duality. Item~(b) i.e.~Infinitesimal Torelli suggests that the weight-$2$ Hodge structure of $X$ might capture much of the geometry of $X$. 

We will review some of the   known results regarding higher-dimensional HK's and then we will present a program which aims to prove that numerical $K3^{[2]}$'s behave very much like $K3$'s at least as far as Items~(1)-(2) and~(3) above are concerned - a HK $4$-fold $X$ is a {\it numerical
 $(K3)^{[2]}$} if there exists
 an isomorphism of abelian groups $\psi\colon
 H^2(X;\ZZ)\overset{\sim}{\lra}H^2(S^{[2]};\ZZ)$ where $S$ is a $K3$ 
 such that
\begin{equation}
 \int_{X}\alpha^{4}=\int_{S^{[2]}}\psi(\alpha)^{4}\qquad \forall 
 \alpha\in H^2(X;\ZZ). 
\end{equation}
  In the last section we will discuss Global Torelli for deformations of $K3^{[2]}$.
 \section{Examples}
 \setcounter{equation}{0}
The surprising topological properties of HK manifolds (see Subsection~\ref{topofthepops}) led Bogomolov~\cite{bog1} to state erroneously that no higher-dimensional (i.e.~of $\dim>2$) HK exists. Some time later Fujiki~\cite{fujiex}
realized that $K3^{[2]}$ is a higher-dimensional HK manifold\footnote{Fujiki described $K3^{[2]}$ not as a Douady space but as the blow-up of the diagonal in the symmetric square of a $K3$ surface}. Soon after that Beauville~\cite{beau} showed that $K3^{[n]}$ is a HK manifold and constructed another deformation class of HK manifolds 
in arbitrary even dimension greater that $2$ namely deformations of generalized Kummers.  We exhibited~\cite{ogprimo,ogsecondo} two \lq\lq sporadic\rq\rq deformation classes, one in dimension $6$ the other in dimension $10$. No other deformation classes are known other than those mentioned above.  
 \subsection{Beauville}\label{esembeau}
\setcounter{equation}{0}
Besides $(K3)^{[n]}$ Beauville discovered another class of $2n$-dimensional HK manifolds - generalized Kummers associated to a $2$-dimensional compact complex torus. Before defining generalized Kummers we recall that the Douady space $W^{[n]}$  comes with a  cycle (Hilbert-Chow) map
 \begin{equation}\label{hilbchow}
\begin{matrix}
W^{[n]} & \overset{\kappa_n}{\lra} & W^{(n)} \\
[Z] & \mapsto & \sum_{p\in W}\ell(\cO_{Z,p})p
\end{matrix}
\end{equation}
where $W^{(n)}$ is the symmetric product of $W$.
 Now suppose that $T$ is a $2$-dimensional compact complex torus. We have the summation map $\sigma_n\colon W^{(n)} \to  W$. Composing the two above maps (with $(n+1)$ replacing $n$) we get a locally (in the classical topology)  trivial fibration $\sigma_{n+1}\circ\kappa_{n+1}\colon W^{[n+1]}\to W$. The $2n$-dimensional {\it generalized Kummer} associated to $T$ is
 \begin{equation}\label{eccokumm}
K^{[n]}T:=(\sigma_{n+1}\circ\kappa_{n+1})^{-1}(0).
\end{equation}
The name is justified by the observation that if $n=1$ then $K^{[1]}T$ is the Kummer surface associated to $T$ (and hence a $K3$). Beauville~\cite{beau} proved that $K^{[n]}(T)$ is a HK manifold. Moreover  if $n\ge 2$ then
\begin{equation}
  b_2((K3)^{[n]})=23\qquad b_2(K^{[n]}T)=7.
\end{equation}
 In particular $(K3)^{[n]}$ and $K^{[n]}T$ are not deformation equivalent as soon as $n\ge 2$. The second cohomology of these manifolds is described as follows. Let $W$ be a compact complex surface. There is a \lq\lq symmetrization map\rq\rq
\begin{equation}
\mu_n\colon H^2(W;\ZZ)  \lra  H^2(W^{(n)};\ZZ) 
\end{equation}
characterized by the following property. Let $\rho_n\colon W^n\to W^{(n)}$ be the quotient map and $\pi_i\colon W^n\to W$ be the $i$-th projection: then
\begin{equation}\label{discesa}
\rho_n^{*}\circ\mu_n(\alpha)=\sum_{i=1}^n\pi_i^{*}\alpha,\qquad \alpha\in H^2(W;\ZZ).
\end{equation}
Composing with $\kappa_n^{*}$ and extending scalars one gets an injection of integral Hodge structures
\begin{equation}\label{simcom}
\wt{\mu}_n:=\kappa_n^{*}\circ\mu_n
\colon H^2(W;\CC)  \lra  H^2(W^{[n]};\CC). 
\end{equation}
The above map is not surjective unless $n=1$; we are missing the Poincar\'e dual of the exceptional set of $\kappa_n$ i.e.
\begin{equation}
\Delta_n:=\{[Z]\in W^{[n]}\mid \text{$Z$ is non-reduced}\}.
\end{equation}
It is known that $\Delta_n$ is a prime divisor and that it is divisible\footnote{If $n=2$ Equation~\eqref{doppio} follows from existence of the double cover $Bl_{diag}(S^2)\to S^{[2]}$ ramified over $\Delta_2$.} by $2$ in $\Pic (W^{[n]})$:
\begin{equation}\label{doppio}
\cO_{W^{[n]}}(\Delta_n)\cong L_n^{\otimes 2},\quad L_n\in \Pic (W^{[n]}).
\end{equation}
Let $\xi_n:=c_1(L_n)$; one has
\begin{equation}\label{comhilb}
H^2(W^{[n]};\ZZ)=\wt{\mu}_n H^2(W;\ZZ)\oplus \ZZ\xi_n,\quad\text{if $H_1(W)=0$.} 
\end{equation}
That describes $H^2((K3)^{[n]})$. Beauville proved that an analogous result holds for generalized Kummers, namely we have an isomorphism
\begin{equation}
\begin{matrix}
H^2(T;\ZZ)\oplus\ZZ & \overset{\sim}{\lra} & H^2(K^{[n]}T;\ZZ) \\
(\alpha,k) & \mapsto & (\wt{\mu}_{n+1}(\alpha)+k\xi_{n+1})|_{K^{[n]}T}
\end{matrix}
\end{equation}
 The above description of the $H^2$  gives the following interesting result: if $n\ge 2$ the  generic deformation of $S^{[n]}$ where $S$ is a $K3$  is not isomorphic to $T^{[n]}$ for some other $K3$ surface $T$.
In fact every deformation of $S^{[n]}$ obtained by deforming $S$ keeps $\xi_n$ of type $(1,1)$ while as  noticed previously   the generic deformation of a HK manifold has no non-trivial integral $(1,1)$-classes.  (Notice that if  $S$ is a surface of general type then every deformation of $S^{[n]}$ is indeed obtained by deforming $S$, see~\cite{fan}.)
 \subsection{Mukai and beyond}
\setcounter{equation}{0}
 Mukai~\cite{muksympl,mukaivb,mukaisug} and Tyurin~\cite{tyu} analyzed moduli spaces of semistable sheaves on projective $K3$'s and abelian surfaces and obtained other examples of HK manifolds. Let $S$ be a projective $K3$ and $\cM$ the moduli space of $\cO_S(1)$-semistable sheaves on $S$ with assigned Chern character - by Gieseker and Maruyama $\cM$ has a natural structure of projective scheme. A non-zero canonical form  
   on $S$ induces a holomorphic symplectic $2$-form on the open $\cM^s\subset\cM$ parametrizing stable sheaves (notice that $\cM^s$ is smooth by Mukai~\cite{muksympl}). 
   If $\cM^s=\cM$ then $\cM$ is a HK variety\footnote{A HK variety is a projective HK manifold.}, in general it is not isomorphic (nor birational) to $(K3)^{[n]}$ however it can be  deformed to   $(K3)^{[n]}$ (here $2n=\dim\cM$), see~\cite{go-huy,ogvb,yoshi1}. Notice that $S^{[n]}$ may be viewed as a particular case of Mukai's construction by identifying it with the moduli space of rank-$1$ semistable sheaves on $S$ with $c_1=0$ and $c_2=n$. Notice also that these moduli spaces give explicit deformations of $(K3)^{[n]}$ which are not $(K3)^{[n]}$. 
   Similarly one may consider moduli spaces of semistable sheaves on an abelian  surface $A$: in the case when $\cM=\cM^s$ one gets deformations of the generalized Kummer. To be precise it is not $\cM$ which is a deformation of a generalized Kummer but rather one of its Beauville-Bogomolov factors. Explicitely we consider the map
   \begin{equation}\label{fattore}
\begin{matrix}
\cM(A) & \overset{\mathfrak a}{\lra} & A\times \wh{A}\\
[F] & \mapsto & (\alb (c_2(F)-c_2(F_0)),[\det F\otimes (\det F_0)^{-1}])
\end{matrix}
\end{equation}
where $[F_0]\in\cM$ is a \lq\lq reference\rq\rq point and $\alb \colon CH_0(A)\to A$ is the Albanese map. Then $\mathfrak a$ is a locally (classical topology) trivial fibration; Yoshioka~\cite{yoshi2} proved that the fibers of ${\mathfrak a}$ are deformations of a generalized Kummer. What can we say about moduli spaces such that 
   $\cM\not=\cM^s$ ?  The locus $(\cM\setminus\cM^s)$  parametrizing $S$-equivalence classes of semistable non-stable sheaves is the singular locus of $\cM$ except for pathological choices of Chern character which do not give anything particularly interesting; thus we assume that $(\cM\setminus\cM^s)$ is the singular locus of $\cM$. A natural question is the following: does there exist a crepant desingularization $\wt{\cM}\to\cM$ ?  We constructed such a desingularization~\cite{ogprimo,ogsecondo} (see also~\cite{chma}) for the moduli space $\cM_4(S)$ 
 of semi-stable rank-$2$ sheaves on a $K3$ surface $S$ with $c_1=0$ and $c_2=4$ and  for the moduli space $\cM_2(A)$ of semi-stable sheaves on an abelian surface $A$ with $c_1=0$ and $c_2=2$; the singularities of the moduli spaces are the same in both cases and both moduli spaces have dimension $10$. Let $M_{10}$ be our desingularization of $\cM_4(S)$ where $S$ is a $K3$. Since the resolution is crepant Mukai's holomorphic symplectic form on $(\cM(S)\setminus\cM(S)^s)$ extends to a holomorphic symplectic form on $M_{10}$. We proved~\cite{ogprimo} that $M_{10}$ is HK i.e.~it is simply connected and $h^{2,0}(M_{10})=1$. Moreover $M_{10}$ is not a deformation of one of Beauville's examples because $b_2(M_{10})=24$. (We proved that $b_2(M_{10})\ge 24$ later Rapagnetta~\cite{rap2} proved that equality holds.) Next let $A$ be an abelian surface and $\wt{\cM}_2(A)\to \cM_2(A)$ be our desingularization. 
   Composing Map~\eqref{fattore} for $\cM(A)=\cM_2(A)$ with the desingularization  map we get a  locally (in the classical topology) trivial fibration $
\wt{\mathfrak a}\colon \wt{\cM}_2(A)\to A\times \wh{A}$; let $M_6$ be any fiber of $\wt{\mathfrak a}$. We proved~\cite{ogprimo} that $M_6$ is  HK and that $b_2(M_6)=8$; thus $M_6$  is not a deformation of one of Beauville's examples. We would like to point out that while all Betti and Hodge numbers of Beauville's examples are known~\cite{gottsche1} the same is not true of our examples (Rapagnetta~\cite{rap1} computed the Euler characteristic of $M_6$). Of course there are examples of moduli spaces $\cM$ with $\cM\not=\cM^s$ in any even dimension; one would like to  desingularize them  and  produce many more deformation classes of HK manifolds. Kaledin-Lehn-Sorger~\cite{kls} have proved that in most cases there is no crepant desingularization and that if there is one then it is a deformation of $M_{10}$ if the surface is a $K3$ while in the case of an  abelian  surface the fibers of Map~\eqref{fattore} composed with the desingularization map are deformations of $M_6$\footnote{To be precise their result holds if the polarization of the surface is \lq\lq generic\rq\rq relative to the chosen Chern character, with this hypothesis  the singular locus of $\cM$ is, so to speak, as small as possible}. In fact all known examples of HK manifolds are deformations either of Beauville's examples or of ours.     
 \subsection{Mukai flops}\label{chirurgia}
\setcounter{equation}{0}
Let $X$ be a HK manifold of dimension $2n$ containing a submanifold $Z$ isomorphic to $\PP^n$. The {\it Mukai flop of $Z$} (introduced in~\cite{muksympl})  is a bimeromorphic map $X\dashrightarrow X^{\vee}$ which is an isomorhism away from $Z$ and replaces $Z$ by the dual plane $Z^{\vee}:=(\PP^n)^{\vee}$. Explicitly  let $\tau\colon\wt{X}\to X$ be the blow-up of $Z$ and $E\subset \wt{X}$ be the exceptional divisor. Since $Z$ is Lagrangian the symplectic form on $X$ defines an isomorphism $N_{Z/X}\cong \Omega_Z=\Omega_{\PP^n}$. 
Thus we have
\begin{equation}
E\cong\PP(N_{Z/X})=\PP(\Omega_{\PP^n})\subset\PP^n\times(\PP^n)^{\vee}.
\end{equation}
Hence $E$ is a $\PP^{n-1}$-fibration in two different ways: we have $\pi\colon E\to\PP^{n}$ i.e.~the restriction of $\tau$ to $E$  
 and $\rho\colon E\to(\PP^{n})^{\vee}$. A straightforward computation shows that the restriction of $N_{E/\wt{X}}$ to a fiber of $\rho$ is $\cO_{\PP^{n-1}}(-1)$. By the Fujiki-Nakano contractibility criterion there exists a proper map $\tau^{\vee}\colon \wt{X}\to X^{\vee}$ to a complex manifold $X^{\vee}$ which is an isomorphism outside $E$ and which restricts to $\rho$ on $E$. Clearly $\tau^{\vee}(E)$ is naturally identified with $Z^{\vee}$ and we have a bimeromorphic map $X\dashrightarrow X^{\vee}$ which defines an isomorphism $(X\setminus Z)\overset{\sim}{\lra} (X^{\vee}\setminus Z^{\vee})$. 
Summarizing: we have  the following commutative diagram
\begin{equation}\label{contraggo}
\xymatrix{ 
 & \wh{X} \ar_{\tau}[dl]\ar^{\tau^{\vee}}[dr] & \\
X \ar_{c}[dr]   & \dra & X^{\vee} \ar^{c^{\vee}}[dl] \\   
& W & \\}
\end{equation}
where $c\colon X\to W$ and $c^{\vee}\colon X^{\vee}\to W$ are the contractions of $Z$ and $Z^{\vee}$ respectively - see the Introduction of~\cite{wiwi}. 
 It follows that $X^{\vee}$ is simply connected and a holomorphic symplectic form on $X$ gives a holomorphic symplectic form on $X^{\vee}$ spanning $H^0(\Omega^2_{X^{\vee}})$; thus $X^{\vee}$ is HK if it is K\"ahler. We give an example with $X$ and $X^{\vee}$ projective. Let $f\colon S\to\PP^2$ be a double cover branched over a smooth sextic and $\cO_{S}(1):=f^{*}\cO_{\PP^2}(1)$: thus $S$ is a $K3$ of degree $2$. Let $X:=S^{[2]}$ and $\cM$ be the moduli space of pure $1$-dimensional $\cO_S(1)$-semistable sheaves on $S$  with typical member $\iota_{*}\cL$ where $\iota\colon C\hra S$ is the inclusion of $C\in |\cO_S(1)|$ and $\cL$ is a line-bundle on $C$ of degree $2$. We have a natural rational map 
 \begin{equation}\label{noniso}
\phi\colon S^{[2]}\dashrightarrow \cM
\end{equation}
 which associates to $[W]\in S^{[2]}$ the sheaf  $\iota_{*}\cL$ where $C$ is the unique curve containing $W$ (unicity requires $W$ to be generic !) and $\cL:=\cO_C(W)$. If every divisor in $ |\cO_S(1)|$ is prime (i.e.~the branch curve of $f$ has no tritangents) then $\cM$ is smooth (projective) and the rational map $\phi$ is identified with the flop of  
 \begin{equation}
Z:=\{f^{-1}(p)\mid p\in\PP^2\}.
\end{equation}
Wierzba and Wi\'sniewsky~\cite{wiwi} have proved that any birational map between HK four-folds is a composition of Mukai flops.
In higher dimensions Mukai~\cite{muksympl}  defined more general flops in which the indeterminacy locus is a fibration in projective spaces. Markman~\cite{mark0} constructed {\it stratified Mukai flops}. 
 \section{General theory}
\setcounter{equation}{0}
 It is fair to state that there are three main ingredients in the general theory of HK manifolds developed by Bogomolov, Beauville, Fujiki, Huybrechts and others:
 \begin{itemize}
\item[(1)]
Deformations are unobstructed (Bogomolov's Theorem).
\item[(2)]
The canonical Bogomolov-Beauville quadratic form on $H^2$ of a HK manifold (see the next subsection).
\item[(3)]
Existence of the twistor family on a HK manifold equipped with a K\"ahler class: this is a  consequence of Yau's solution of Calabi's conjecture.
\end{itemize}
 \subsection{Topology}\label{topofthepops}
\setcounter{equation}{0}
 Let $X$ be a HK-manifold of dimension $2n$. Beauville~\cite{beau} and Fujiki~\cite{fuji} proved that there exist an integral indivisible quadratic form 
 \begin{equation}
 q_X\colon H^2(X)\to\CC
\end{equation}
 (cohomology is with complex coefficients) and $c_X\in\QQ_{+}$ such that
 \begin{equation}\label{belleq}
\int_X\alpha^{2n}=c_X\frac{(2n)!}{n! 2^n}q_X(\alpha)^n,\qquad \alpha\in H^2(X).
\end{equation}
The above equation determines $c_X$ and $q_X$ with no ambiguity unless $n$ is even. If $n$ is even then $q_X$ is determined up to $\pm 1$: one singles out one of the two choices by imposing the inequality 
 \begin{equation}\label{bellin}
\text{$q_X(\omega)>0$ for $\omega\in H^{1,1}_{\RR}(X)$  a K\"ahler class.}
\end{equation}
(Notice that if $n$ is odd the above inequality follows from~\eqref{belleq}.) The {\it Beauville-Bogomolov} form and the {\it Fujiki constant} of $X$  are $q_X$ and $c_X$  respectively. We notice that~\eqref{belleq} is equivalent (by polarization) to 
\begin{equation}\label{polarizzo}
\int_X\alpha_1\wedge\ldots\wedge\alpha_{2n}=c_X
\sum_{\sigma\in\cR_{2n}} (\alpha_{\sigma(1)},\alpha_{\sigma(2)})_X\cdot
(\alpha_{\sigma(3)},\alpha_{\sigma(4)})_X\cdots
(\alpha_{\sigma(2n-1)},\alpha_{\sigma(2n)})_X
\end{equation}
where $(\cdot,\cdot)_X$ is the symmetric bilinear form associated to $q_X$ and $\cR_{2n}$ is a set of representatives for the left cosets of the subgroup $\cG_{2n}<\cS_{2n}$ of  permutations  of $\{1,\ldots,2n\}$  generated by transpositions $(2i-1,2i)$ and by products of transpositions $(2i-1,2j-1)(2i,2j)$ - in other words in the right-hand side of~\eqref{polarizzo} we avoid repeating addends which are equal\footnote{In defining $c_X$ we have introduced a normalization which is not standard in order to avoid a combinatorial factor in Equation~\eqref{polarizzo}.}. The existence of $q_X, c_X$ is by no means trivial; we sketch a proof. Let $f\colon\cX\to T$ be a deformation of $X$ representing $Def(X)$; more precisely letting $X_t:=f^{-1}\{t\}$ for $t\in T$, we are given $0\in T$, an isomorphism $X_0\overset{\sim}{\lra} X$ and the induced map of germs $(T,0)\to Def(X)$ is an isomorphism. In particular $T$ is smooth in $0$ and hence we may assume that it is a polydisk. The Gauss-Manin connection defines    
an integral isomorphism $ \phi_t\colon H^2(X)\overset{\sim}{\lra} H^2(X_t).$
The {\it local period map} of $X$ is given by
\begin{equation}\label{mappaperiodi}
  \begin{matrix}
 T & \overset{\pi}{\lra} & \PP(H^2(X)) \\
 t & \mapsto & \phi_t^{-1}H^{2,0}(X_t)
  \end{matrix}
\end{equation}
By infintesimal Torelli, see~\eqref{derper} $Im\pi$ is an analytic hypersurface in an open (classical topology) neighborhood of $\pi(0)$ and hence its Zariski closure $V=\ov{Im\pi}$ is either all of $\PP(H^2(X))$ or a hypersurface. One shows that the latter holds by considering the (non-zero) degree-$2n$ homogeneous polynomial
\begin{equation}\label{forminter}
  \begin{matrix}
 H^2(X) & \overset{G}{\lra} & \CC \\
 \alpha & \mapsto & \int_X\alpha^{2n}
  \end{matrix}
\end{equation}
 In fact if $\sigma_t\in H^{2,0}(X_t)$ then 
 \begin{equation}\label{prodiere}
 \int_{X_t}\sigma_t^{2n}=0
\end{equation}
 by type consideration and it follows by Gauss-Manin parallel transport that $G$ vanishes on $V$. 
 Thus $I(V)=(F)$ where $F$ is an irreducible homogeneous polynomial. By considering the derivative of the period map~\eqref{derper} one checks easily that $V$ is not a hyperplane and hence $\deg F\ge 2$. 
 On the other hand type consideration gives something stronger than~\eqref{prodiere}, namely 
\begin{equation}\label{tipozero}
  \int_{X_t}\sigma_t^{n+1}\wedge\alpha_1\cdots\wedge\alpha_{n-1}=0\qquad
\alpha_1,\ldots,\alpha_{n-1} \in H^2(X_t). 
\end{equation}
It follows that all the derivatives of $G$ up to order $(n-1)$ included vanish on $V$.  Since $\deg G=2n$ and $\deg F\ge 2$ it follows that $G=c\cdot F^n$ and $\deg F=2$. By integrality of $G$ there exists $\lambda\in\CC^{*}$ such that $c_X:=\lambda c$ is rational positive, $q_X:=\lambda\cdot F$ is integral indivisible and~\eqref{belleq} is satisfied.

\vskip 2mm
\n
 Of course if $X$ is a $K3$ then $q_X$ is the intersection form of $X$ (and $c_X=1$). In general $q_X$ gives $H^2(X;\ZZ)$ a structure of lattice just as in the well-known case of $K3$ surfaces. Suppose that  $X$ and $Y$ are deformation equivalent HK-manifolds: it follows from~\eqref{belleq} that $c_X=c_Y$ and the lattices $H^2(X;\ZZ),H^2(Y;\ZZ)$ are isometric (see the comment following~\eqref{belleq} if $n$ is even). Consider the case when $X=(K3)^{[n]}$; then $\wt{\mu}_n$ is an isometry, $\xi_n\bot Im\wt{\mu}_n$ and $q_X(\xi_n)=-2(n-1)$ i.e.
\begin{equation}\label{rethilb}
H^2(S^{[n]};\ZZ)\cong  U^3\widehat{\oplus}E_8\la-1\ra^2
 \widehat{\oplus}\la -2(n-1)\ra
\end{equation}
where $\widehat{\oplus}$ denotes othogonal direct-sum, $U$ is the hyperbolic plane and   $E_8\la-1\ra$ is the unique rank-$8$ negative definite unimodular even lattice. Moreover the Fujiki constant is 
\begin{equation}\label{fujimori}
c_{S^{[n]}}=1.
\end{equation}
In~\cite{rap2} the reader will find the B-B quadratic form and Fujiki constant of the  other known deformation classes of HK manifolds. 
\begin{rmk}\label{rmk:quix}
 Let $X$ be a HK manifold of dimension $2n$ and $\omega\in H^{1,1}_{\RR}(X)$ be a K\"ahler class.
 \begin{itemize}
\item[(1)]
Equation~\eqref{belleq} gives that with respect to $(,)_X$ we have
 \begin{equation}\label{tipoqu}
\text{$H^{p,q}(X)\bot H^{p',q'}(X)$ unless $(p',q')=(2-p,2-q)$.}
\end{equation}
\item[(2)]
 $q_X(\omega)>0$. In fact let $\sigma$ be  generator of $H^{2,0}(X)$; by Equation~\eqref{polarizzo} and Item~(1) above we have
\begin{equation}
0<\int_X \sigma^{n-1}\wedge\ov{\sigma}^{n-1}\wedge\omega^2=
c_X(n-1)! (\sigma,\ov{\sigma})_X q_X(\omega).
\end{equation}
Since $c_X>0$ and $(\sigma,\ov{\sigma})_X>0$ we get that $q_X(\omega)>0$ as claimed. 
\item[(3)]
The index  of $q_X$ is $(3,b_2(X)-3)$ (i.e.~that is the index of its restriction   to $H^2(X;\RR)$). In fact applying Equation~\eqref{polarizzo} to $\alpha_1=\ldots=\alpha_{2n-1}=\omega$ and arbitrary $\alpha_{2n}$ we get that $\omega^{\bot}$ is equal to the primitive cohomology  $H^2_{pr}(X)$ (primitive with respect to $\omega$). On the other hand 
Equation~\eqref{polarizzo} with $\alpha_1=\ldots=\alpha_{2n-2}=\omega$ and  $\alpha_{2n-1},\alpha_{2n}\in\omega^{\bot}$ gives that a positive multiple of $q_X|_{\omega^{\bot}}$ is equal to the standard quadratic form on $H^2_{pr}(X)$ (recall Inequality~\eqref{bellin}). By the Hodge index Theorem it follows that the restriction of $q_X$ to $\omega^{\bot}\cap H^2(X;\RR)$ has index $(2,b_2(X)-3)$. Since $q_X(\omega)>0$ it follows that 
$q_X$ has index $(3,b_2(X)-3)$.
\item[(4)]
Let $D$  be an effective divisor on $X$; then $(\omega,D)_X>0$. (Of course   $(\omega,D)_X$ denotes $(\omega,c_1(\cO_X(D)))_X$.) In fact the inequality follows from the inequality $\int_D\omega^{2n-1}>0$ together with~\eqref{polarizzo} and Item~(2) above. 
\item[(5)]
Let $f\colon X\dashrightarrow Y$ be a birational map  where $Y$ is a HK manifold. Since $X$ and $Y$ have trivial canonical bundle $f$ defines an isomorphism $U\overset{\sim}{\lra} V$ where $U\subset X$ and $V\subset Y$ are open sets with complements of codimension at least $2$. It follows that $f$ induces an isomorphism $f^{*}\colon H^2(Y;\ZZ)\overset{\sim}{\lra} H^2(X;\ZZ)$; $f^{*}$ is an isometry of lattices, see Lemma~2.6 of~\cite{huy}. 
\end{itemize}
\end{rmk}
The proof of existence of $q_X$ and $c_X$ may be adapted to prove the following useful generalization of~\eqref{belleq}.
\begin{prp}\label{prp:piatta}
 Let $X$ be a HK manifold of dimension $2n$. Let $\cX\to T$ be a representative of the deformation space of $X$. Suppose  that $0\not=\gamma\in H^{p,p}_{\RR}(X)$ is a class which remains of type $(p,p)$ under Gauss-Manin parallel transport (e.~g.~the Chern class $c_p(X)$). Then $p$ is even and moreover there exists $c_{\gamma}\in\RR$ such that 
 \begin{equation}
\int_X\gamma\wedge\alpha^{2n-p}=c_{\gamma}q_X(\alpha)^{n-p/2}.
\end{equation}
\end{prp}
Our next topic is Verbitsky's theorem~\cite{verbo} (see also~\cite{bog2}). Let $X$ be a HK-manifold of dimension $2n$.  Our (sketch of) proof of~\eqref{belleq} shows that 
\begin{equation}\label{svanisce}
\text{if $\alpha\in H^2(X)$ and $q_X(\alpha)=0$ then $\alpha^{n+1}=0$ in $H^{2n+2}(X)$.}
\end{equation}
In fact  adopting the notation introduced in the proof of~\eqref{belleq} we have $0=\sigma_t^{n+1}\in H^{2n+2}(X_t)$ and hence by Gauss-Manin transport we get that $0=(\psi_t^{-1}\sigma_t)^{n+1}\in H^{2n+2}(X)$. Since the set  $\{\psi_t^{-1}\sigma_t\mid t\in T\}$ is Zariski dense in the zero-set $V(q_X)\subset H^2(X)$ we get~\eqref{svanisce}. Let $I\subset Sym^{\bullet}H^2(X)$ be the ideal generated by $\alpha^{n+1}$ where $\alpha\in H^2(X)$ and $q_X(\alpha)=0$:
\begin{equation}
I:=\la \{\alpha^{n+1}\mid \alpha\in H^2(X),\quad q_X(\alpha)=0\} \ra.
\end{equation}
By~\eqref{svanisce} we have a natural map of $\CC$-algebras
\begin{equation}\label{surf}
Sym^{\bullet}H^2(X)/I\longrightarrow H^{\bullet}(X).
\end{equation}
\begin{thm}{\rm[Verbitsky]}\label{thm:spracht}
Map~\eqref{surf} is injective.
\end{thm} 
In particular we get that cup-product defines an injection
\begin{equation}\label{inietto}
\bigoplus_{q=0}^n Sym^q H^2(X)\hra H^{\bullet}(X).
\end{equation}
 S.~M.~Salamon~\cite{sal} proved that there is a non-trivial linear constraint on the Betti numbers of a compact K\"ahler manifold  carrying a holomorphic symplectic form (for example a HK manifold); the proof consists in a clever application of the Hirzebruch-Riemann-Roch formula to the sheaves $\Omega^p_X$ and the observation that the symplectic form induces an isomorphism $\Omega^p_X\cong \Omega^{2n-p}_X$ where $2n=\dim X$\footnote{A non-zero section of the canonical bundle defines an isomorphism $\Omega^{2n-p}_X\cong (\Omega^p_X)^{\vee}=\wedge^ p T_X$ and the symplectic form defines an isomorphism $T_X\cong \Omega_X$ and hence $\wedge^ p T_X\cong \Omega^p_X$}.
 \begin{thm}{\rm[S.~M.~Salamon]}
Let $X$ be a compact K\"ahler manifold of dimension $2n$ carrying a holomorphic symplectic form. Then
\begin{equation}\label{vincolo}
nb_{2n}(X)=2\sum_{i=1}^{2n}(-1)^i(3i^2 -n)b_{2n-i}(X).
\end{equation}
\end{thm}
The following corollary of Verbitsky's and Salamon's results was obtained by Beauville (unpublished) and Guan~\cite{guan}.
 \begin{crl}{\rm[Beauville and Guan]}\label{crl:ventitre}
Let $X$ be a HK $4$-fold. Then $b_2(X)\le 23$. If equality holds then $b_3(X)=0$ and 
moreover the map
\begin{equation}\label{prodsim}
Sym^2 H^2(X;\QQ)\lra H^4(X;\QQ)
\end{equation}
induced by cup-product is an isomorphism.
\end{crl}
\begin{proof}
Let $b_i:=b_i(X)$. Salamon's equation~\eqref{vincolo} for $X$ reads
\begin{equation}\label{michelle}
b_4=46+10 b_2-b_3.
\end{equation}
By Verbitsky's~\Ref{thm}{spracht} - see~\eqref{inietto} - we have 
\begin{equation}\label{livorno}
{b_2+1\choose 2}\le b_4.
\end{equation}
Replacing $b_4$ by the right-hand side of~\eqref{michelle} we get that
\begin{equation}
b_2^2+b_2\le 92+20 b_2 -2 b_3\le 92+20 b_2.
\end{equation}
It follows that $b_2\le 23$ and that if equality holds then $b_3=0$. Suppose that $b_2=23$: then  $b_4=276$ by~\eqref{michelle} and hence~\eqref{prodsim} follows from Verbitsly's~\Ref{thm}{spracht}.
\end{proof}
  We mention that Guan~\cite{guan} obtained other restrictions on $b_2(X)$ for a HK four-fold $X$: for example $8<b_2(X)<23$ is \lq\lq forbidden\rq\rq.

 \subsection{The K\"ahler cone}
\setcounter{equation}{0}
Let $X$ be a HK manifold of dimension $2n$.  The convex cone $\cK_X\subset H^{1,1}_{\RR}(X)$  of K\"ahler classes is the {\it K\"ahler cone of $X$}. 
Item~(3) of~\Ref{rmk}{quix} gives that the
restriction of $q_X$ to $H^{1,1}_{\RR}(X)$ is non-degenerate of signature $(1,b_2(X)-3)$; it follows that the cone
\begin{equation}\label{quadpos}
\{\alpha\in H^{1,1}_{\RR}(X)\mid q_X(\alpha)>0\}
\end{equation}
 has two connected components. 
By Item~(2) of~\Ref{rmk}{quix}  $\cK_X$ is contained in~\eqref{quadpos}. Since $\cK_X$ is  convex   it is contained in a single connected component of~\eqref{quadpos}; that component is the {\it positive cone} $\cC_X$. The following result is proved in the erratum of~\cite{huy}.
\begin{thm}{\rm [Huybrechts]}\label{thm:kalpos}
 Let $X$ be a HK manifold. Let $\cX\to T$ be a representative of $Def(X)$ with $T$ irreducible. If $t\in T$ is very general (i.e.~outside a countable union of proper analytic subsets of $T$) then 
 \begin{equation}\label{rinogaetano}
\cK_{X_t}=\cC_{X_t}. 
\end{equation}
\end{thm}
\begin{proof}
Let $0\in T$ be the point such that $X_0\cong X$ and the induced map of germs $(T,0)\to Def(X)$ is an isomorphism\footnote{The map $(T,0)\to Def(X)$ depends on the choice of an isomorphism $f\colon X_0\overset{\sim}{\lra} X$ but whether it is an isomorphism or not is independent of $f$.}.
By shrinking $T$ around $0$ if necessary we may assume that $T$ is simply connected and that $\cX\to T$ represents $Def(X_t)$ for every $t\in T$.
In particular the Gauss-Manin connection gives an isomorphism
$P_t\colon H^{\bullet}(X;\ZZ)\overset{\sim}{\lra}H^{\bullet}(X_t;\ZZ)$ for every $t\in T$. Given $\gamma\in H^{2p}(X;\ZZ)$ we let 
\begin{equation}
T_{\gamma}:=\{t\in T\mid \text{$P_t(\gamma)$ is of type $(p,p)$}\}. 
\end{equation}
Let
\begin{equation}\label{fuorida}
t\in (T\setminus\bigcup_{T_{\gamma}\not=T}T_{\gamma})
\end{equation}
and $Z\subset X_t$ be a closed analytic subset of codimension $p$; we claim that  
\begin{equation}\label{positivo}
\int_{Z}\alpha^{2n-p}>0\quad \text{if $q_{X_t}(\alpha)>0$.}
\end{equation}
In fact let $\gamma\in H^{p,p}_{\RR}(X_t)$ be the Poincar\'e dual of $Z$. By~\eqref{fuorida}   $\gamma$ remains of type $(p,p)$ for every deformation of $X_t$; by~\Ref{prp}{piatta} $p$ is even and moreover there exists $c_{\gamma}\in\RR$ such that 
\begin{equation}\label{mina}
\int_Z\alpha^{2n-p}=c_{\gamma}q_X(\alpha)^{n-p/2}\qquad\forall \alpha\in H^2(X_t).
\end{equation}
Let $\omega$ be a K\"ahler class. Since $0<\int_Z\omega^{2n-p}$  and $0<q_X(\omega)$ we get that $c_{\gamma}>0$; thus~\eqref{positivo} follows from~\eqref{mina}.
Now apply Demailly-Paun's version of the Nakai-Moishezon ampleness criterion~\cite{dempau}: $\cK_{X_t}$ is a connected component of the set $P(X_t)\subset H^{1,1}_{\RR}(X_t)$ of  classes $\alpha$ such that $\int_{Z}\alpha^{2n-p}>0$ for all closed analytic subsets $Z\subset X_t$ (here $p=\cod(Z,X_t)$). Let $t$ be as in~\eqref{fuorida}. 
By~\eqref{positivo} $P(X_t)=\cC_{X_t}\coprod(-\cC_{X_t})$; since  $\cK_{X_t}\subset\cC_{X_t}$ we get the proposition.
\end{proof}
Huybrechts~\cite{huy} has proved that~\Ref{thm}{kalpos} gives   the following {\it projectivity criterion}.
\begin{thm}{\rm[Huybrechts]}\label{thm:procri}
A HK manifold $X$ is projective if and only if there exists a (holomorphic) line-bundle $L$ on $X$ such that $q_X(c_1(L))>0$.  
\end{thm}
Boucksom~\cite{bouck}, elaborating on ideas of Huybrechts, gave the following characterization of $\cK_X$ for  arbitary $X$. 
\begin{thm}{\rm[Boucksom]}
Let $X$ be a HK manifold. A class $\alpha\in H^{1,1}_{\RR}(X)$ is K\"ahler if and only if it belongs to the positive cone $\cC_X$ and moreover $\int_C\alpha>0$ for every rational curve $C$\footnote{A curve is rational if it is irreducible and its normalization is rational}.
\end{thm}
 One would like to have a numerical description of the K\"ahler (or ample) cone as in the $2$-dimensional case. Hassett and Tschinkel~\cite{hasstsch2} proved the following result.
 \begin{thm}{\rm[Hassett - Tschinkel]}\label{thm:ampiose}
Let $X$ be a HK variety deformation equivalent to $K3^{[2]}$ and $L_0$ an ample line-bundle on $X$.  Let $L$ be a line-bundle on $X$ such that $c_1(L)\in\cC_X$.
Suppose that $(c_1(L),\alpha)_X>0$ for all $\alpha\in H^{1,1}_{\ZZ}(X)$ such that 
$(c_1(L_0),\alpha)_X>0$ and
\begin{itemize}
\item[(a)] 
$q_X(\alpha)=-2$ or
\item[(b)]
$q_X(\alpha)=-10$ and $(\alpha,H^2(X;\ZZ))_X=2\ZZ$.
\end{itemize}
Then $L$ is ample.
\end{thm}
Hassett and Tschinkel~\cite{hasstsch1} conjectured that  the converse of the above theorem holds i.e.~the above conditions are also necessary for $L$ to be ample.   
We explain the appearance of the conditions in the above theorem and why one expects that the converse holds. We start with Item~(a). Let $X$ be a HK manifold deformation equivalent to $K3^{[2]}$ and $L$ a line-bundle on $X$: Hirzebruch-Riemann-Roch for $X$ reads 
\begin{equation}\label{rrhilbsq}
\chi(L)=\frac{1}{8} (q(L)+4)(q(L)+6).
\end{equation}
(We let $q=q_X$.) It follows that $\chi(L)=1$ if and only if $q(L)=-2$ or $q(L)=-8$. 
\begin{cnj}{\rm [Folk?]}\label{cnj:diveff}
 Let $X$ be a HK manifold deformation equivalent to $K3^{[2]}$. Let $L$ be a line-bundle on $X$ such that $q_X(L)=-2$.
 \begin{itemize}
\item[(1)]
If $(c_1(L),H^2(X;\ZZ))_X=\ZZ$  then either $L$ or $L^{-1}$ has a non-zero section. 
\item[(2)]
If $(c_1(L),H^2(X;\ZZ))_X=2\ZZ$  then either $L^2$ or $L^{-2}$ has a non-zero section. (Notice that $q_X(L^{\pm 2})=-8$.)
\end{itemize}
\end{cnj}
If the above conjecture holds then given $\alpha\in H^{1,1}_{\ZZ}(X)$ with $q_X(\alpha)=-2$ we have that either $(\alpha,\cdot)_X$ is strictly positive or strictly negative on $\cK_X$; in particular the condition corresponding to Item~(a) of~\Ref{thm}{ampiose} is necessary for a line-bundle to be ample. 
Below are examples of line-bundles satisfying Items~(1), (2) above. 
  \begin{itemize}
\item[Ex.~1]
  Let $S$ be a $K3$ containing a smooth rational curve $C$ and $X=S^{[2]}$. Let
\begin{equation}
D:=\{[Z]\in S^{[2]}\mid Z\cap C\not=\es\}.
\end{equation}
Let $L:=\cO_X(D)$; then $c_1(L)=\wt{\mu}_2(c_1(\cO_S(C)))$ where $\wt{\mu}_2$ is given by~\eqref{simcom}.
Since $\wt{\mu}_2$ is an isometry we have $q_X(L)=C\cdot C=-2$ and moreover $(c_1(L),H^2(X;\ZZ))_X=\ZZ$. For another example see Item~(5) of~\Ref{rmk}{fibcon}
\item[Ex.~2]
Let $S$ be a $K3$  and $X=S^{[2]}$. Let $L_2$ be the square-root of $\cO_X(\Delta_2)$ where $\Delta_2\subset S^{[2]}$ is the divisor parametrizing non-reduced subschemes - thus $c_1(L_2)=\xi_2$. Then $q(L_2)=-2$ and $L_2^{2}$ has \lq\lq the\rq\rq non-zero section vanishing on $\Delta_2$. Notice that neither $L_2$ nor $L_2^{-1}$ has a non-zero section.
\end{itemize}
Summarizing: line-bundles of square $-2$ on a HK deformation of $K3^{[2]}$ should be similar to $(-2)$-classes on a K3. (Recall that if $L$ is  a line-bundle on a $K3$ with $c_1(L)^2=-2$ then by  Hirzebruch-Riemann-Roch and Serre duality either $L$ or $L^{-1}$ has a non-zero section.)       Next we explain Item~(b) of~\Ref{thm}{ampiose}. 
 Suppose that $X$ is a HK deformation of  $K3^{[2]}$ and that $Z\subset X$ is a closed submanifold isomorphic to $\PP^2$ - see Section~\ref{chirurgia}. Let $C\subset Z$ be a line. Since $(,)_X$ is non-degenerate (but not unimodular !) there exists $\beta\in H^2(X;\QQ)$ such that 
 \begin{equation}
 \int_C\gamma=(\beta,\gamma)_X\qquad\forall \gamma\in H^2(X).
\end{equation}
One proves that 
\begin{equation}\label{cinmez}
q_X(\beta)=-\frac{5}{2}.
\end{equation}
Equation~\eqref{cinmez} follows from Isomorphism~\eqref{prodsim} and the good properties of deformations of HK manifolds, see~\cite{hasstsch2}, Sect.~4. Since $(\beta,H^2(X;\ZZ))_X=\ZZ$ and the discriminant of $(,)_X$ is $2$ we have $2\beta\in H^2(X;\ZZ)$; thus $\alpha:=2\beta$ is as in Item~(b) of~\Ref{ampiose} and if $L$ is ample then $0<\int_C c_1(L)=\frac{1}{2}(c_1(L),\alpha)_X$.

\vskip  2mm
\n
Hassett and Tschinkel state conjectures that extend~\Ref{thm}{ampiose} and its converse to general HK varieties, see~\cite{hasstsch3} - in particular they give a conjectural numerical description of the effective cone of a HK variety. The 
papers~\cite{bouck2,druel} contain key results in this circle of ideas.

\vskip  2mm
\n
We close the section by stating a beautiful result of Huybrechts~\cite{huykahler} - the proof is based on results on the K\"ahler cone and uses in an essential way the existence of the  twistor family.
\begin{thm}\label{thm:birdef}
 Let $X$ and $Y$ be bimeromorphic HK manifolds. Then $X$ and $Y$ are deformation equivalent.
\end{thm}
 \section{Complete families of HK varieties}
\setcounter{equation}{0}
A  couple $(X,L)$ where $X$ is a HK variety and $L$ is a primitive\footnote{i.e.~$c_1(L)$ is indivisibile in $H^2(X;\ZZ)$.} ample line-bundle on $X$ with $q_X(L)=d$ is 
a {\it HK variety of degree $d$}; an isomorphism $(X,L)\overset{\sim}{\lra}(X',L')$   between HK's of degree $d$ consists of an isomorphism $f\colon X\overset{\sim}{\lra} X'$ such that $f^{*}L'\cong L$.  A family of HK varieties of degree $d$ is a couple
\begin{equation}\label{famiglia}
(f\colon \cX\to T,\,\cL)
\end{equation}
 where $\cX\to T$ is a family of HK varieties deformation equivalent to a fixed HK manifold $X$ and $\cL$ is a line-bundle such that $(X_t, L_t)$ is a HK variety of degree $d$ for every $t\in T$ (here $X_t:=f^{-1}(t)$ and $L_t:=\cL|_{X_t}$) - we say that it is a family of HK varieties if we are not intersted in the value of $q_X(L_t)$. The  deformation space of $(X,L)$ is a codimension-$1$ smooth sub-germ $Def(X,L)\subset Def(X)$ with tangent space the kernel of Map~\eqref{ostia} with $\alpha= c_1(L)$. The family~\eqref{famiglia}  is {\it locally complete} if given any $t_0\in T$ the map of germs $(T,t_0)\to Def(X_{t_0},L_{t_0})$ is surjective, it is {\it globally complete} if given any HK variety  $(Y,L)$ of degree $d$ with $Y$ deformation equivalent to $X$ there exists $t_0\in T$ such that $(Y,L)\cong (X_{t_0},L_{t_0})$. In dimension $2$ i.e.~for $K3$ surfaces one has explicit globally complete families of low degree: If $d=2$ the family of double covers $S\to\PP^2$ branched over a smooth sextic will do\footnote{In order to get a global family we must go to a suitable double cover of the parameter space of sextic curves.}, if $d=4$ we may consider the family of  smooth quartic surfaces $S\subset\PP^3$  with the addition of certain \lq\lq limit\rq\rq surfaces (double covers of smooth quadrics and certain elliptic $K3$'s) corresponding to degenerate quartics (double quadrics and the surface swept out by tangents to a rational normal cubic curve respectively).  The list goes on for quite a few values of $d$, see~\cite{mukaiexpls,mukai13}  and then it necessarily stops - at least in this form - because moduli spaces of high-degree $K3$'s are not unirational~\cite{ghsk3}. We remark that in low degree one shows \lq\lq by hand\rq\rq that there exists a globally complete family which is irreducible; the same is true in arbitrary degree but I know of no elementary proof, the most direct argument is via Global Torelli. What is the picture in higher ($>2$) dimensions ? Four distinct (modulo obvious equivalence)  locally complete families of higher-dimensional HK varieties have been constructed - they are all deformations of $K3^{[2]}$. The families are the following:
\begin{itemize}
\item[(1)]
We constructed~\cite{og2} the family of double covers of certain special sextic  hypersurfaces in $\PP^5$ that we named EPW-sextics (they had been introduced by Eisenbud-Popescu-Walter~\cite{epw}). The polarization is the pull-back of  $\cO_{\PP^5}(1)$; its degree is $2$.
\item[(2)]
Let $Z\subset\PP^5$ be a smooth cubic hypersurface; Beauville and Donagi~\cite{beaudon} proved that the variety parametrizing lines on $Z$ is a deformation of $K3^{[2]}$. The polarization  is given by the Pl\"ucker embedding: it has degree $6$.
\item[(3)]
Let $\sigma$ be a generic $3$-form on $\CC^{10}$; Debarre and Voisin~\cite{debvoi} proved that the set $Y_{\sigma}\subset Gr(6,\CC^{10})$ parametrizing subspaces on which $\sigma$ vanishes is a deformation of $K3^{[2]}$. The polarization is given by the Pl\"ucker embedding: it has degree $22$.
\item[(4)]
Let $Z\subset\PP^5$ be a generic cubic hypersurface; Iliev and Ranestad~\cite{iliran1,iliran2} have proved that the variety of sums of powers $VSP(Z,10)$ \footnote{$VSP(Z,10)$ parametrizes $9$-dimensional linear spaces of  $|\cO_{\PP^5}(3)|$ which contain $Z$ and are $10$-secant to the Veronese $\{[L^3]\mid L\in (H^0(\cO_{\PP^5}(1))\setminus\{0\})\}$.} is a deformation of $K3^{[2]}$.  For the polarization we refer to~\cite{iliran2}; the degree is $38$ (unpublished computation by Iliev, Ranestad and Van Geemen). 
\end{itemize}
For each of the above families -  more precisely for the family obtained by adding \lq\lq limits\rq\rq - one might ask whether it is globally complete for HK varieties of the given degree which are deformations of $K3^{[2]}$. As formulated the answer is negative with the possible exception of our family, for a  trivial reason: in the lattice $L:=H^2(K3^{[2]};\ZZ)$ the orbit of a primitive vector $v$ under the action of $O(L)$ is determined by the value of the  B-B form $q(v)$ plus the extra information on whether 
 \begin{equation}
(v,L)=
\begin{cases}
\ZZ & or\\
2\ZZ & 
\end{cases}
\end{equation}
In the first case one says that the {\it divisibility of $v$} is $1$, in the second case that it is $2$; if the latter occurs then $q(v)\equiv 6 \pmod{8}$. Thus the divisibility of the polarization in Item~(1) above equals $1$; on the other hand it equals $2$ for the  families  in Item~(2)-(4). The correct question regarding  global completeness is the following. Let  $X$ be a HK deformation of $K3^{[2]}$ with an ample line-bundle $L$ such that either $q(L)=2$ or $q(L)\in\{6,22,38\}$ and the divisibility of $c_1(L)$ is equal to $2$: does there exist a variety $Y$ parametrized by one of the above families - or a limit of such -  and an isomorphism $(X,L)\cong (Y,\cO_Y(1))$ ? If a \lq\lq naive\rq\rq global Torelli holds for HK deformations of $K3^{[2]}$ then the answer is positive,  see~\Ref{clm}{sciarra}. 

None of the families above is as easy to construct as  are the families of low-degree $K3$ surfaces. There is the following Hodge-theoretic explanation. In order to get a locally complete family of varieties  one  usually constructs complete intersections (or sections of ample vector-bundles) in homogeneous varieties: by Lefschetz' hyperplane Theorem such a construction will never produce a higher-dimensional HK.  On the other hand the families of Items~(1), (2) and~(3) are related to complete intersections as follows (I do not know whether one may view the Iliev-Ranestad family from a similar perspective). First if $f\colon X\to Y$ is a double EPW-sextic (Item~(1) above) then $f$ is the quotient map of an involution $ X\to X$ which has one-dimensional $(+1)$-eigenspace on $H^2(X)$ - in particular it kills $H^{2,0}$ - and \lq\lq allows\rq\rq the quotient to be a hypersurface. Regarding Item~(2): let $Z\subset\PP^5$ be a smooth cubic hypersurface and $X$ the variety of lines on $Z$, the incidence correspondence in $Z\times X$ induces an isomorphism of the primitive Hodge structures $H^4(Z)_{pr}\overset{\sim}{\to} H^2(X)_{pr}$. Thus a Tate twist of $H^2(X)_{pr}$ has become the primitive intermediate cohomology  of a hypersurface. A similar comment applies to the Debarre-Voisin family (and there is a similar incidence-type construction of double EPW-sextics given by Iliev and Manivel~\cite{iliman}).

 In this section we will describe in some detail the family of double EPW-sextics and we will say a few words about analogies with the Beauville-Donagi  family. 
 \subsection{Double EPW-sextics, I}
\setcounter{equation}{0}
We start by giving the definition of
EPW-sextic~\cite{epw}.
Let $V$ be a $6$-dimensional complex vector space. We choose a volume-form $vol\colon\wedge^6 V\overset{\sim}{\lra}\CC$ and  we equip $\wedge^3 V$ with the symplectic form  
\begin{equation}\label{maitresse}
(\alpha,\beta)_V:=\vol(\alpha\wedge\beta). 
\end{equation}
Let $\lagr$ be the symplectic Grassmannian parametrizing Lagrangian subspaces of $\wedge^3 V$ - notice that $\lagr$ is independent of the chosen volume-form $vol$. 
Given a non-zero $v\in V$ we let 
\begin{equation}
F_v:=\{\alpha\in\wedge^3 V\mid v\wedge\alpha=0\}.
\end{equation}
Notice that $(,)_V$ is zero on $F_v$ and   $\dim(F_v)=10$ i.e.~$F_v\in \lagr$. Let 
\begin{equation}\label{eccoeffe}
F\subset\wedge^3 V\otimes\cO_{\PP(V)}
\end{equation}
be  the sub-vector-bundle with fiber $F_v$ over $[v]\in\PP(V)$. Given $A\in\lagr$ we let
\begin{equation}
Y_A=\{[v]\in\PP(V)\mid F_v\cap A\not=\{0\}\}.
\end{equation}
Thus $Y_A$ is the degeneracy locus of the map
\begin{equation}\label{diecidieci}
F\overset{\lambda_A}{\lra}(\wedge^3 V/A)\otimes\cO_{\PP(V)}
\end{equation}
where $\lambda_A$ is given by Inclusion~\eqref{eccoeffe} 
followed by the quotient map $\wedge^3 V\otimes\cO_{\PP(V)}\to 
(\wedge^3 V/A)\otimes\cO_{\PP(V)}$.  Since the vector-bundles appearing in~\eqref{diecidieci} have equal rank  $Y_A$ is the zero-locus of $\det\lambda_A\in H^0(\det F^{\vee})$ - in particular it has a natural structure of closed subscheme of $\PP(V)$. 
A straightforward computation gives that
$\det F\cong\cO_{\PP(V)}(-6)$ and hence $Y_A$ is a sextic hypersurface unless it equals $\PP(V)$\footnote{Given $[v]\in\PP(V)$ there exists $A\in\lagr$ such that $A\cap F_v=\{0\}$ and hence $[v]\notin Y_A$; thus $Y_A$ is a sextic hypersurface for generic $A\in\lagr$. On the other hand if $A=F_w$ for some $[w]\in\PP(V)$ then $Y_A=\PP(V)$.}; if the former holds we say that $Y_A$ is  an {\it EPW-sextic}. What do EPW-sextics look like? The main point  is that locally they are the degeneracy locus of a symmetric map of vector-bundles (they were introduced by Eisenbud, Popescu and Walter to give examples of a \lq\lq quadratic sheaf\rq\rq, namely    $\coker(\lambda_A)$, which can not be expressed {\bf globally} as the cokernel of a symmetric map of vector-bundles on $\PP^5$). More precisely given 
 $B\in\lagr$  we let   $\cU_B\subset\PP(V)$  be the open subset defined by
\begin{equation}\label{calub}
\cU_B:=\{[v]\in\PP(V)\mid  F_v\cap B=\{0\}\}\,.
\end{equation}
Now choose $B$ transversal to $A$. We have a direct-sum decomposition $\wedge^3 V=A\oplus B$; since $A$ is lagrangian the symplectic form $(,)_V$ defines  an isomorphism $B\cong A^{\vee}$.
Let  $[v]\in\cU_B$: since $F_{v}$ is transversal to $B$ it is the graph of a  map 
\begin{equation}\label{mappagrafo}
\tau_A^B([v])\colon A\to B\cong A^{\vee},\qquad [v]\in\cU_B\,.
\end{equation}
 The map $\tau_A^B([v])$ is symmetric  because $A,B$ and $F_v$ are lagrangians.
\begin{rmk}\label{rmk:classico}
There is one choice of $B$ which produces a \lq\lq classical\rq\rq description of $Y_A$, namely $B=\wedge^3 V_0$ where $V_0\subset V$ is a codimension-$1$ subspace\footnote{It might happen that there is no $V_0$ such that $\wedge^3 V_0$ is transversal to $A$: in that case $A$ is unstable for the natural $PGL(V)$-action on $\lagr$ and hence we may forget about it.}. With such a choice of $B$ we have $\cU_B=(\PP(V)\setminus\PP(V_0))$; we identify it with $V_0$ by choosing $v_0\in(V\setminus V_0)$ and mapping
\begin{equation}\label{pochina}
\begin{matrix}
 V_0 & \overset{\sim}{\lra} & \PP(V)\setminus\PP(V_0) \\
 v & \mapsto & [v_0+v]
\end{matrix}
\end{equation}
The direct-sum decomposition $\wedge^3 V=F_{v_0}\oplus\wedge^3 V_0$ and transversality $A\pitchfork\wedge^3 V_0$ allows us to view $A$ as the graph of a (symmetric) map $\wt{q}_A\colon F_{v_0}\to\wedge^3 V_0$. Identifying $\wedge^2 V_0$ with $F_{v_0}$ via the isomorphism
 \begin{equation}\label{effevu}
\begin{matrix}
\wedge^2 V_0 & \overset{\sim}{\lra} & F_{v_0} \\
\alpha & \mapsto & v_0\wedge \alpha
\end{matrix}
\end{equation}
we may view $\wt{q}_A$ as a symmetric map 
 \begin{equation}
\wedge^2 V_0\lra \wedge^3 V_0=\wedge^2 V_0^{\vee}.
\end{equation}
We let $q_A\in Sym^2(\wedge^2 V_0^{\vee})$ be the quadratic form  corresponding to $\wt{q}_A$. Given $v\in V_0$ let $q_v\in Sym^2(\wedge^2 V_0^{\vee})$ be the Pl\"ucker quadratic form $q_v(\alpha):=vol(v_0\wedge v\wedge\alpha\wedge\alpha)$. Modulo Identification~\eqref{pochina} we have
\begin{equation}
Y_A\cap(\PP(V)\setminus\PP(V_0))=V(\det(q_A+q_v)).
\end{equation}
Equivalently let
\begin{equation}
Z_A:=V(q_A)\cap \GG r(2,V_0)\subset\PP(\wedge^2 V_0)\cong\PP^9.
\end{equation}
Then we have an isomorphism
\begin{equation}\label{lapunta}
\begin{matrix}
 \PP(V) & \overset{\sim}{\lra} & |\cI_{Z_A}(2)| \\
 [\lambda v_0+\mu v] & \mapsto & V(\lambda q_A+\mu q_v)
\end{matrix}
\end{equation}
(Here $\lambda,\mu\in\CC$ and $v\in V_0$.) Let $D_A\subset |\cI_{Z_A}(2)|$ be the discriminant locus; modulo the above identification we have
\begin{equation}
Y_A\cap(\PP(V)\setminus\PP(V_0))=D_A\cap ( |\cI_{Z_A}(2)|\setminus  
|\cI_{\GG r(2,V_0)}(2)|).
\end{equation}
Notice that $|\cI_{\GG r(2,V_0)}(2)|$ is a hyperplane contained in $D_A$ with multiplicity $4$; that explains why $\deg Y_A=6$ while $\deg D_A=10$.
\end{rmk}
We go back to general considerations regarding $Y_A$. The symmetric map $\tau_A^B$ of~\eqref{mappagrafo} allows us to give a structure of scheme to the degeneracy locus
\begin{equation}
Y_A[k]=\{[v]\in\PP(V)\mid \dim(A\cap F_v)\ge k\}
\end{equation}
by declaring that $Y_A[k]\cap\cU_B=V(\wedge^{(11-k)}\tau_A^B)$. By a standard dimension count we expect that the following holds for generic $A\in\lagr$: $Y_A[3]=\es$, $Y_A[2]=sing Y_A$ and $Y_A[2]$ is a smooth surface (of degree $40$ by~(6.7) of~\cite{fulpra}), in particular $Y_A$ should be a very special sextic hypersurface. This is indeed the case; in order to be less \lq\lq generic\rq\rq  let
\begin{eqnarray}
\Delta:=  &  \{A\in\lagr\mid Y_A[3]\not=\es\}\,,\label{eccodel}\\
\Sigma:= & \{A\in\lagr\mid \text{$\exists W\in{\mathbb G}r(3,V)$ s.~t.~$\wedge^3 W\subset A$}\}.\label{eccosig}
\end{eqnarray}
 A straightforward computation shows that $\Sigma$ and $\Delta$ are distinct closed irreducible  codimension-$1$ subsets of $\lagr$.
 Let 
 \begin{equation}\label{eccozero}
\lagr^0:=  \lagr\setminus\Sigma\setminus\Delta\,.
\end{equation}
Then $Y_A$ has the generic behaviour described  above if and only if it belongs to $\lagr^0$. Next let $A\in\lagr$ and suppose that $Y_A\not=\PP(V)$: then $Y_A$ comes equipped with a natural double cover  $f_A\colon X_A\to Y_A$ defined
 as follows. Let $i\colon Y_A\hra\PP(V)$ be the inclusion map: 
since $\coker(\lambda_A)$ is annihilated by a local generator of $\det\lambda_A$ we have $\coker(\lambda_A)=i_{*}\zeta_A$ for a sheaf $\zeta_A$ on $Y_A$.
Choose $B\in\lagr$ transversal to $A$; the direct-sum decomposition
$\wedge^3 V=A\oplus B$
 defines a projection map  $\wedge^3 V\to A$; thus we get   a map $ \mu_{A,B}\colon F\to A\otimes\cO_{\PP(V)}$. 
 We claim that there is a  commutative diagram with exact rows
\begin{equation}\label{spqr}
\begin{array}{ccccccccc}
0 & \to & F&\mapor{\lambda_A}& A^{\vee}\otimes\cO_{\PP(V)} & \lra & i_{*}\zeta_A
&
\to & 0\\
 & & \mapver{\mu_{A,B}}& &\mapver{\mu^{t}_{A,B}} &
&
\mapver{\beta_{A}}& & \\
0 & \to & A\otimes\cO_{\PP(V)}& \mapor{\lambda_A^{t}}& F^{\vee} & \lra &
Ext^1(i_{*}\zeta_A,\cO_{\PP(V)}) & \to & 0
\end{array}
\end{equation}
(Since $A$ is Lagrangian the symplectic form defines a canonical isomorphism $\left(\wedge^3 V/A\right)\cong A^{\vee}$; that is why we may write $\lambda_A$ as above.)
In fact the second row is obtained by applying the $Hom(\,\cdot\, ,\cO_{\PP(V)})$-functor to the first row and the equality $\mu_{A,B}^t\circ\lambda_A=\lambda_A^t\circ\mu_{A,B}$ holds because
 $F$ is a Lagrangian sub-bundle of $\wedge^3 V\otimes\cO_{\PP(V)}$.
 Lastly $\beta_A$ is defined to be the unique map making the diagram commutative; as suggested by notation it  is independent of  $B$. Next by applying the $Hom(i_{*}\zeta_A,\,\cdot\,)$-functor to the exact sequence
\begin{equation}
0\lra\cO_{\PP(V)}\lra\cO_{\PP(V)}(6)\lra\cO_{Y_A}(6)\lra 0
\end{equation}
 we get the exact sequence
\begin{equation}\label{romaleone}
0\lra i_{*}Hom(\zeta_A,\cO_{Y_A}(6))\overset{\partial}{\lra}
Ext^1(i_{*}\zeta_A,\cO_{\PP(V)})\overset{n}{\lra} Ext^1(i_{*}\zeta_A,\cO_{\PP(V)}(6))
\end{equation}
where $n$ is locally equal to multiplication  by $\det\lambda_A$. Since the second row of~\eqref{spqr} is exact  a local generator of $\det\lambda_A$ annihilates  $Ext^1(i_{*}\zeta_A,\cO_{\PP(V)})$; thus $n=0$ and hence we get a canonical isomorphism
\begin{equation}\label{extugualehom}
\partial^{-1}\colon 
Ext^1(i_{*}\zeta_A,\cO_{\PP(V)})\overset{\sim}{\lra} i_{*}Hom(\zeta_A,\cO_{Y_A}(6)).
\end{equation}
Let
\begin{equation}\label{emmea}
\begin{matrix}
\zeta_A\times\zeta_A & \overset{\wt{m}_A}\lra & \cO_{Y_A}(6) \\
(\sigma_1,\sigma_2) & \mapsto & (\partial^{-1}\circ\beta_A(\sigma_1))(\sigma_2).
\end{matrix}
\end{equation}
Let $\xi_A:=\zeta_A(-3)$;
tensorizing both sides of~\eqref{emmea} by $\cO_{Y_A}(-6)$ we get   a multiplication map 
\begin{equation}
m_A\colon\xi_A\times\xi_A\to\cO_{Y_A}.
\end{equation}
The above  multiplication map equips  $\cO_{Y_A}\oplus\xi_A$ with the structure of a commutative and associative  $\cO_{Y_A}$-algebra. 
We let 
\begin{equation}
X_A:=\Spec(\cO_{Y_A}\oplus\xi_A),\qquad 
f_A\colon X_A\to Y_A.
\end{equation}
Then $X_A$ is a {\it double EPW-sextic}. Let $\cU_B$ be as in~\eqref{calub}: we may describe  $f_A^{-1}(Y_A\cap \cU_B)$  as follows. Let $M$ be the symmetric matrix associated to~\eqref{mappagrafo} by a
choice of  basis of $A$ and   $M^c$ be the matrix of cofactors of $M$. Let $Z=(z_1,\ldots,z_{10})^t$ be the coordinates on $A$ associated to the given basis; then $f_A^{-1}(Y_A\cap \cU_B)\subset \cU_B\times{\mathbb A}^{10}_{Z}$ and its ideal is generated by 
 the entries of the matrices
\begin{equation}
M\cdot Z\,,\ \ Z\cdot Z^t-M^c\,.
\end{equation}
(The \lq\lq missing\rq\rq equation $\det M=0$ follows by Cramer's rule.)  One may reduce the size of $M$ in a neighborhood of $[v_0]\in\cU_B$ as follows. The kernel of the symmetric map $\tau^B_A([v_0])$ equals $A\cap   F_{v_0}$; let $J\subset A$ be complementary to $A\cap   F_{v_0}$. Diagonalizing the restriction of $\tau_A^B$ to $J$ we may assume that 
\begin{equation}
M([v])=
\begin{pmatrix}
 M_0([v]) & 0 \\
 0 & 1_{10-k}
\end{pmatrix}
\end{equation}
where $k:=\dim(A\cap F_{v_0})$
and $M_0$ is a symmetric $k\times k$ matrix. It follows at once that $f_A$ is \'etale over $(Y_A\setminus Y_A[2])$. We also get the following description of $f_A$ over 
 a point $[v_0]\in(Y_A[2]\setminus Y_A[3])$  under the hypothesis that there is no $0\not= v_0\wedge v_1\wedge v_2\in A$. First $f_A^{-1}([v_0])$ is a single point $p_0$, secondly $X_A$ is smooth at $p_0$ and there exists an involution $\phi$ on $(X_A,p_0) $ with $2$-dimensional fixed-point set such that $f_A$ is identified with the quotient map $(X_A,p_0)\to (X_A,p_0)/\la\phi\ra$. It follows that $X_A$ is smooth if $A\in\lagr^0$. We may fit together all smooth double EPW-sextics by going to a suitable double cover $\rho\colon \lagr^\star\to \lagr^0$;  there exist a family of HK four-folds $\cX\to \lagr^{\star}$ and a  relatively ample line-bundle $\cL$ over $\cX$ such that for all $t\in \lagr^\star$ we have
 $(X_t,L_t)\cong (X_{A_t},f_{A_t}^{*}\cO_{Y_{A_t}}(1))$ where
 \begin{equation}
X_t:=\rho^{-1}(t),\quad L_t=\cL|_{X_t},\quad A_t:=\rho(t).
\end{equation}
  The following result was proved in~\cite{og2}.
\begin{thm}{\rm[O'Grady]}\label{thm:epwdoppie}
 Let $A\in\lagr^0$. Then $X_A$ is a HK four-fold deformation equivalent to $K3^{[2]}$. Moreover $\cX\to \lagr^\star$ is a locally complete family of HK varieties of degree $2$.
\end{thm}
\n
{\it Sketch of proof following~\cite{ogemma}.\/} The main issue is to prove that $X_A$ is a HK deformation of $K3^{[2]}$. In fact once this is known the equality
\begin{equation}
\int_{X_A}f_A^{*} c_1(\cO_{Y_A}(1))^4=2\cdot 6=12
\end{equation}
together with~\eqref{belleq} gives that $q(f_A^{*}c_1( \cO_{Y_A}(1)))=2$ and moreover the family $\cX\to \lagr^\star$ is locally complete by the following argument. First Kodaira vanishing and Formula~\eqref{rrhilbsq} give that
\begin{equation}
h^0(f_A^{*} \cO_{Y_A}(1))=\chi(f_A^{*} \cO_{Y_A}(1))=6
\end{equation}
and hence the map
\begin{equation}
X_A\overset{f_A}{\lra} Y_A\hra\PP(V)
\end{equation}
may be identified with the map $X_A\to |f_A^{*} \cO_{Y_A}(1)|^{\vee}$. From this one gets that the natural map $(\lagr^0// PGL(V),[A])\to Def(X_A,f_A^{*} \cO_{Y_A}(1))$ is injective. One concludes that $\cX\to \lagr^\star$ is locally complete by
a dimension count:
\begin{equation}
\dim(\lagr^0// PGL(V))=20=\dim Def(X_A,f_A^{*} \cO_{Y_A}(1)).
\end{equation}
Thus we are left with the task of proving  that $X_A$ is a HK deformation of $K3^{[2]}$ if $A\in\lagr^0$. We do this by analyzing $X_A$ for 
\begin{equation}
A\in(\Delta\setminus\Sigma).
\end{equation}
By definition $Y_A[3]$ is non-empty; one shows that it is finite, that $sing X_A= f_A^{-1} Y_A[3]$ and that $f_A^{-1}[v_i]$ is a single point for each $[v_i]\in Y_A[3]$. There exists a small resolution
\begin{equation}
\pi_A\colon \wh{X}_A\lra X_A,\qquad 
(f_A\circ\pi_A)^{-1}([v_i])\cong\PP^2\quad\forall [v_i]\in Y_A.
\end{equation}
In fact one gets that locally over the points of $sing X_A$ the above resolution is identified with the contraction $c$ (or $c^{\vee}$) appearing in~\eqref{contraggo} - in particular $\wh{X}_A$ is not unique, in fact there are $2^{|Y_A[3]|}$ choices involved in the construction of $\wh{X}_A$. The resolution $\wh{X}_A$ fits into a simultaneous resolution i.e.~given a sufficiently small open (in the classical topology) $A\in U\subset(\lagr\setminus\Sigma)$ we have proper maps $\pi,\psi$
\begin{equation}
\wh{\cX}_U\overset{\pi}{\lra}\cX_U \overset{\psi}{\lra} U
\end{equation}
where $\psi$ is a tautological family of double EPW-sextics over $U$ i.e.~$\psi^{-1}A\cong X_A$ and $(\psi\circ\pi)^{-1}A\to \psi^{-1}A=X_A$ is a small resolution as above if $A\in U\cap\Delta$ while $\pi^{-1}A\cong X_A$ if $A\in(U\setminus\Delta)$. Thus it suffices to prove that there exist $A\in(\Delta\setminus\Sigma)$ such that $\wh{X}_A$ is a HK deformation of $K3^{[2]}$. Let $[v_i]\in Y_A[3]$; we define a $K3$ surface $S_A(v_i)$  as follows. There exists a codimension-$1$ subspace $V_0\subset V$ not containing $v_i$ and such that $\wedge^3 V_0$ is transversal to $A$. Thus $Y_A$ can be  described as in~\Ref{rmk}{classico}: we adopt notation introduced in that remark, in particular we have the quadric $Q_A:=V(q_A)\subset\PP(\wedge^2 V_0)$. The singular locus of $Q_A$ is $\PP(A\cap F_{v_i})$ - we recall Identification~\eqref{effevu}. By hypothesis $\PP(A\cap F_{v_i})\cap\GG r(2,V_0)=\es$; it follows that $\dim\PP(A\cap F_{v_i})=2$ (by hypothesis $\dim\PP(A\cap F_{v_i})\ge 2$). Let
\begin{equation}
S_A(v_i):=Q_A^{\vee}\cap \GG r(2,V_0)\subset \PP(\wedge^2 V_0^{\vee}).
\end{equation}
Then $S_A(v_i)\subset\PP(Ann\,(A\cap F_{v_i}))\cong\PP^6$ is the transverse intersection of a smooth quadric and the Fano $3$-fold of index $2$ and degree $5$, i.e.~the generic $K3$ of genus $6$. There is a natural degree-$2$ rational map
\begin{equation}
g_i\colon S_A(v_i)^{[2]}\dashrightarrow |\cI_{S_A(v_i)}(2)|^{\vee}
\end{equation}
 which associates to $[Z]$ the set of quadrics in $ |\cI_{S_A(v_i)}(2)|$ which contain the line spanned by $Z$ - thus $g_i$ is regular if $S_A(v_i)$ contains no lines. One proves that $Im(g_i)$ may be identified with $Y_A$; it follows that there exists a birational map  
 \begin{equation}
h_i\colon S_A(v_i)^{[2]}\dashrightarrow\wh{X}_A
\end{equation}
Moreover if $S_A(v_i)$ contains no lines (that is true for generic $A\in(\Delta\setminus\Sigma)$) there is a choice of small resolution $\wh{X}_A$ such that $h_i$ is regular and hence an isomorphism - in particular $\wh{X}_A$ is projective\footnote{There is no reason a priori why $\wh{X}_A$ should be K\"ahler, in fact one should expect it to be non-K\"ahler for some $A$ and some choice of small resolution}. This proves that $X_A$ is a HK deformation of $K3^{[2]}$ for $A\in\lagr^0$. 
\qed
\vskip 2mm 
\begin{rmk}\label{rmk:fibcon}
The above proof of~\Ref{thm}{epwdoppie} provides a description of $X_A$ for $A\in(\Delta\setminus\Sigma)$; what about $X_A$ for $A\in\Sigma$? One proves that if $A\in\Sigma$ is generic - in particular there is a unique $W\in\GG r(3,V)$ such that $\wedge^3 W\subset A$ - then the following hold:
\begin{itemize}
\item[(1)]
$C_{W,A}:=\{[v]\in \PP(W)\mid \dim(A\cap F_v)\ge 2\}$ is a smooth sextic curve.
\item[(2)]
$sing X_A=f_A^{-1}\PP(W)$ and the restriction of $f_A$ to $sing X_A$ is the  double cover of $\PP(W)$ branched over $C_{W,A}$, i.e.~a $K3$ surface of degree $2$.  
\item[(3)]
If $p\in sing X_A$ the germ  $(X_A,p)$ (in the classical topology) is isomorphic to the product of a smooth $2$-dimensional germ and an  $A_1$ singularity; thus the blow-up $\wt{X}_A\to X_A$ resolves the singularities of $X_A$. 
\item[(4)]
Let $U\subset\lagr $ be a small open (classical topology) subset containing $A$.
After a  base change $\wt{U}\to U$ of order $2$ branched over $U\cap \Sigma$ there is a simultaneous resolution of singularities of the tautological family of double EPW's parametrized by $\wt{U}$. It follows that $\wt{X}_A$ is a HK deformation of $K3^{[2]}$. 
\item[(5)]
Let $E_A$ be the exceptional divisor  of the blow-up  $\wt{X}_A\to X_A$ and $e_A\in H^2(\wt{X}_A;\ZZ)$ be its Poincar\'e dual; then $q(e_A)=-2$ and $(e_A,H^2(\wt{X}_A;\ZZ))=\ZZ$.  
\end{itemize}
\end{rmk}
 \subsection{The Beauville-Donagi family}\label{bedofam}
\setcounter{equation}{0}
Let $\cD,\cP\subset|\cO_{\PP^5}(3)|$ be the prime divisors parametrizing singular cubics and cubics containing a plane respectively. We recall that if $Z\in(|\cO_{\PP^5}(3)|\setminus \cD)$ then 
\begin{equation}
X=F(Z):=\{L\in\GG r(1,\PP^5)\mid L\subset X\}
\end{equation}
is a HK four-fold deformation equivalent to $K3^{[2]}$. Let $H$ be the Pl\"ucker ample divisor  on $X$ and $h=c_1(\cO_X(H))$; then
\begin{equation}
q(h)=6,\qquad (h,H^2(X;\ZZ))_X=2\ZZ.
\end{equation}
These results are proved in~\cite{beaudon} by considering the codimension-$1$ locus of Pfaffian cubics; they show that if $Z$ is a generic such Pfaffian cubic then $X$ is isomorphic to $S^{[2]}$ where $S$ is a $K3$ of genus $8$ that one associates to $Z$, moreover the class $h$ is identified with $2\wt{\mu}(D)-5\xi_2$ where $D$ is the class of the (genus $8$) hyperplane class of $S$. Here we will stress the similarities between the HK four-folds parametrized by $\cD,\cP$ and those parametrized by the loci $\Delta,\Sigma\subset\lagr$ described in the previous subsection. Let $Z\in\cD$ be generic. Then $Z$ has a unique singular point $p$ and it is ordinary quadratic, moreover the set of lines in $Z$ containing $p$ is a $K3$ surface $S$ of genus $4$. The variety $X=F(Z)$ parametrizing lines in $Z$ is  birational to $S^{[2]}$; the birational  map is given by
\begin{equation}
\begin{matrix}
S^{[2]} & \dashrightarrow & F(Z) \\
\{L_1,L_2\} & \mapsto  & R
\end{matrix}
\end{equation}
where $L_1+L_2+R=\la L_1,L_2\ra\cdot Z$. 
Moreover $F(Z)$ is singular with singular locus  equal to $S$. 
Thus from this point of view $\cD$ is similar to $\Delta$. On the other hand let $Z_0\in(|\cO_{\PP^5}(3)|\setminus \cD)$ be \lq\lq close\rq\rq to $Z$; 
the monodromy action on $H^2(F(Z_0))$ of a loop in $(|\cO_{\PP^5}(3)|\setminus \cD)$ which goes once  around $\cD$ has order $2$ and hence as far as monodromy is concerned $\cD$ is similar to $\Sigma$.
(Let $U\subset |\cO_{\PP^5}(3)|$ be a small open (classical topology) set containing $Z$; 
it is natural to expect that after a base change $\pi\colon\wt{U}\to U$ of order $2$ ramified over $\cD$ the family of $F(Z_u)$ for $u\in (\wt{U}\setminus \pi^{-1}\cD)$ can be completed over points of $\pi^{-1}\cD$ with HK four-folds birational (isomorphic?) to $S^{[2]}$.) Now let $Z\in\cP$ be generic, in particular it contains a unique plane $P$. Let $T\cong\PP^2$ parametrize $3$-dimensional linear supspaces of $\PP^5$ containing $P$; given $t\in T$ and $L_t$ the corresponding $3$-space  the intersection $L_t\cdot Z$ decomposes as $P+Q_t$ where $Q_t$ is a quadric surface. Let $E\subset X=F(Z)$ be the set defined by
\begin{equation}
E:=\{L\in F(Z)\mid \text{$\exists t\in T$ such that $L\subset Q_t$}\}. 
\end{equation}
For $Z$ generic we have a well-defined map $E\to T$ obtained by associating to $L$ the unique $t$ such that $L\subset Q_t$; the Stein factorization of $E\to T$  is $E\to S\to T$ where $S\to T$ is the double cover ramified over the curve $B\subset T$ parametrizing singular quadrics. The locus $B$ is a smooth sextic curve and hence $S$ is a $K3$ surface of genus $2$. The picture is: $E$ is a conic bundle over the $K3$ surface $S$ and we have
\begin{equation}
q(E)=-2,\qquad (e,H^2(X;\ZZ))=\ZZ,\quad e:=c_1(\cO_X(E)).
\end{equation}
Thus from this point of view $\cP$ is similar to $\Sigma$ - of course if we look at monodromy the analogy fails.
 \section{Numerical Hilbert squares}
 \setcounter{equation}{0}
A {\it numerical Hilbert square} is a HK four-fold $X$ such that $c_X$ is equal to the Fujiki constant of $K3^{[2]}$ and  the lattice $H^2(X;\ZZ)$  is isometric to 
$H^2(K3^{[2]};\ZZ)$; by~\eqref{rethilb}, \eqref{fujimori} this holds if and only if
\begin{equation}
H^2(X;\ZZ)\cong U^3\wh{\oplus} E_8(-1)\wh{\oplus}\la -2\ra,\qquad
c_X=1.
\end{equation}
  We will present a program which aims to prove that a numerical Hilbert square is a deformation of $K3^{[2]}$ i.e.~an analogue of Kodaira's theorem that any two $K3$'s are deformation equivalent. First we recall how  Kodaira~\cite{kod} proved that $K3$ surfaces form a single deformation class. Let $X_0$ be a $K3$. Let $\cX\to T$ be a representative of the deformation space $Def(X_0)$. The image of the local period map $\pi\colon T\to \PP(H^2(X_0))$  contains an open (classical topology) subset of the quadric $\cQ:=V(q_{X_0})$. The set $\cQ(\QQ)$ of rational points of $\cQ$ is  dense (classical topology) in the set of real points $\cQ(\RR)$; it follows that the image $\pi(T)$ contains a point $[\sigma]$ such that $\sigma^{\bot}\cap H^2(X_0;\QQ)$ is generated by a non-zero $\alpha$ such that $q_X(\alpha)=0$. Let $t\in T$ such that $\pi(t)=[\sigma]$ and set $X:=X_t$; by the Lefschetz $(1,1)$ Theorem we have
  \begin{equation}
H^{1,1}_{\ZZ}(X)=\ZZ c_1(L),\qquad q_X(c_1(L))=0
\end{equation}
where $L$ is a holomorphic line-bundle on $X$. 
By Hirzebruch-Riemann-Roch and Serre duality we get that $h^0(L)+h^0(L^{-1})\ge 2$. Thus we may assume that $h^0(L)\ge 2$. It follows that $L$ is globally generated, $h^0(L)=2$ and the map $\phi_L\colon X\to |L|\cong\PP^1$ is an elliptic fibration. Kodaira then proved that any two elliptic $K3$'s are deformation equivalent. J.~Sawon~\cite{saw1} has launched a similar program with the goal of classifying deformation classes of higher-dimensional HK manifolds\footnote{One should assume that $b_2\ge 5$ in order to ensure that the set of rational points in $V(q_X)$  is non-empty (and hence dense in the set of real points).} by deforming them to Lagrangian fibrations - we notice that Matsushita~\cite{mats1,mats2,mats3} has proved quite a few results on HK manifolds which have non-trivial fibrations. The  program is quite ambitious; it runs immediately into the  problem of proving that if $L$ is a non-trivial line-bundle on a  HK manifold $X$ with $q_X(c_1(L))=0$ then $h^0(L)+h^0(L^{-1})>0$\footnote{Let $\dim X=2n$. Hirzebruch-Riemann-Roch gives that $\chi(L)=n+1$, one would like to show that $h^q(L)=0$ for $0<q<2n$.} On the other hand Kodaira's theorem on $K3$'s may be proved~\cite{lepotier} by deforming $X_0$ to a $K3$ surface $X$ such that $H^{1,1}_{\ZZ}(X)=\ZZ c_1(L)$ where $L$ is a holomorphic line-bundle such that $q_X(L)$ is a small positive integer, say $2$. By Hirzebruch-Riemann-Roch and Serre duality $h^0(L)+h^0(L^{-1})\ge 3$ and hence we may assume that $h^0(L)\ge 3$; it follows easily that $L$ is globally generated, $h^0(L)=3$ and the map $\phi_L\colon X\to |L|^{\vee}\cong\PP^2$ is a double cover ramified over a smooth sextic curve. Thus every $K3$ is deformation equivalent to a double cover of $\PP^2$ ramified over a sextic; since the parameter space for smooth sextics is connected it follows that    any two $K3$ surfaces are deformation equivalent. Our idea is to extend this proof to the case of numerical Hilbert squares. In short the plan is as follows. Let $X_0$ be a numerical Hilbert square. First we deform $X_0$ to a HK four-fold $X$ such that 
\begin{equation}\label{quadue}
H^{1,1}_{\ZZ}(X)=\ZZ c_1(L),\qquad q_X(c_1(L))=2
\end{equation}
 and the Hodge structure of $X$ is very generic given the constraint~\eqref{quadue}, see Section~\ref{specializzo} for the precise conditions.  By Huybrechts' Projectivity Criterion~\ref{procri} we may assume that $L$ is ample and then Hirzebruch-Riemann-Roch together with Kodaira vanishing gives that $h^0(L)=6$. Thus we must study the map $f\colon X\dashrightarrow|L|^{\vee}\cong\PP^5$. We prove that either   $f$ is the natural double cover of an EPW-sextic or else it is birational onto its image (a hypersurface of degree at most $12$). We conjecture that the latter never holds; if the conjecture is true then any numerical Hilbert square is a deformation of a double EPW-sextic and hence is a deformation of $K3^{[2]}$.  
 \subsection{The deformation}\label{specializzo}
\setcounter{equation}{0}
We recall Huybrechts' Theorem on surjectivity of the global period map for HK manifolds. 
Let $X_0$ be a HK manifold. Let $L$ be a lattice isomorphic to the lattice $H^2(X_0;\ZZ)$; we denote by $(,)_L$ the extension to $L\otimes\CC$ of the bilinear symmetric form on $L$. The
period domain $\Omega_L\subset\PP(L\otimes \CC)$ is
given by
\begin{equation}\label{perdom}
    \Omega_L:=\{[\sigma]\in \PP(L\otimes \CC)\mid (\sigma,\sigma)_{L}=0,
    \quad (\sigma,\ov{\sigma})_{L}>0\}.
\end{equation}
A  HK manifold $X$ deformation equivalent to $X_0$ is {\it marked} if it is equipped with 
  an isometry of lattices  $\psi\colon L \overset{\sim}{\lra} H^2(X;\ZZ)$. Couples
$(X,\psi)$ and $(X',\psi')$ are equivalent if
there exists an isomorphism $f\colon X\to X'$ such
that $H^2(f)\circ\psi' =\pm\psi$.
The moduli space $\cM_{X_0}$ of
marked HK manifolds
deformation equivalent to $X_0$  is the set of equivalence classes of couples as above.
 If $t\in\cM_{X_0}$ we
let $(X_t,\psi_t)$ be a representative of $t$. Choosing a representative $\cX\to T$ of the deformation space of $X_t$ with $T$ contractible we may put a natural structure 
of (non-separated) complex
analytic manifold on $\cM_{X_0}$, see for example Thm.(2.4) of~\cite{looi}.  The period map is given by
\begin{equation}\label{succofrutta}
\begin{matrix}
  \cM_{X_0} & \overset{\cP}{\lra} & \Omega_{L} \\
  (X,\psi) & \mapsto & \psi^{-1}H^{2,0}(X). 
\end{matrix}
\end{equation}
(we denote by the same symbol both the
isometry $L \overset{\sim}{\lra}
H^2(X;\ZZ)$ and its linear extension
$L\otimes\CC\to H^2(X;\CC)$.) 
The map $\cP$ is
locally an isomorphism by infinitesimal Torelli and local surjectivity of the period map. The following result is proved in~\cite{huy}; the proof  is an adaptation of Todorov's proof of surjectivity for $K3$ surfaces~\cite{tod}.
\begin{thm}{\rm [Todorov, Huybrechts]}\label{thm:tuttiper}
Keep notation as above and let
$\cM^0_{X_0}$ be a connected component of $\cM_{X_0}$.
The restriction of $\cP$ to $\cM^0_{X_0}$ is surjective.
\end{thm}
 Let 
 \begin{equation}
\Lambda:=U^3\widehat{\oplus}E_8\la-1\ra^2
 \widehat{\oplus}\la -2\ra
\end{equation}
be the Hilbert square lattice, see~\eqref{rethilb}. Thus $\Omega_{\Lambda}$ is the period space for numerical Hilbert squares. A straightforward computation gives the following result, see Lemma~3.5 of~\cite{og3}.
\begin{lmm}\label{lmm:eccacca}
Suppose that $\alpha_1,\alpha_2\in\Lambda$
satisfy
\begin{equation}\label{intermatrix}
(\alpha_1,\alpha_1)_{\Lambda}=(\alpha_2,\alpha_2)_{\Lambda}=2,\qquad
(\alpha_1,\alpha_2)_{\Lambda}\equiv 1\mod{2}.
\end{equation}
Let $X_0$ be a numerical Hilbert square. Let $\cM^0_{X_0}$ be
a connected component of the moduli space of marked HK 
four-folds deformation equivalent to $X_0$.
There exists $1\le i\le 2$  such that for every
 $t\in\cM^0_{X_0}$ the class of $\psi_t(\alpha_i)^2$ in
 $H^4(X_t;\ZZ)/Tors$ is indivisible.
\end{lmm}
Notice that $\Lambda$ contains (many) couples $\alpha_1,\alpha_2$ which satisfy~\eqref{intermatrix}; it follows that there exists $\alpha\in\Lambda$ such that  for every
 $t\in\cM^0_{X_0}$ the class of $\psi_t(\alpha)^2$ in
 $H^4(X_t;\ZZ)/Tors$ is indivisible. There exists $[\sigma]\in\Omega_{\Lambda}$ such that  
 \begin{equation}\label{persig}
 \sigma^{\bot}\cap\Lambda=\ZZ\alpha.
\end{equation}
 By~\Ref{thm}{tuttiper} there exists $t\in\cM_{X_0}$ such that $\cP(t)=[\sigma]$; Equality~\eqref{nersev} gives that
\begin{equation}\label{nersev}
H^{1,1}_{\ZZ}(X_t)=\ZZ\alpha.
\end{equation}
Since $q(\psi_t(\alpha))=2>0$ the HK manifold $X_t$ is projective by~\Ref{thm}{procri}; by~\eqref{nersev} either $\psi_t(\alpha)$ or   $\psi_t(-\alpha)$ is ample and hence we may assume that $\psi_t(\alpha)$ is ample. Let $X':=X_t$ and $H'$ be the divisor class such that $c_1(\cO_{X'}(H'))=\psi_t(\alpha)$; $X'$ is a first approximation to the deformation of $X_0$ that we will consider. The reason for requiring that $\psi_t(\alpha)^2$ be indivisible in
 $H^4(X_t;\ZZ)/Tors$ will become apparent in the sketch of the proof of~\Ref{thm}{mainthm1}.   
 \begin{rmk}\label{rmk:casodefo}
If $X$ is a deformation of $K3^{[2]}$ and $\alpha\in H^2(X;\ZZ)$ is an arbitrary class such that $q(\alpha)=2$ then the class of $\alpha^2$ in $H^4(X;\ZZ)/Tors$ is not divisible, see Proposition~3.6 of~\cite{og3}. 
\end{rmk}
Let $\pi\colon\cX\to S$ be a representative of the deformation space $Def(X',H')$. Thus letting $X_s:=\pi^{-1}(s)$ there exist $0\in S$ and a given isomorphism $X_0\overset{\sim}{\lra} X'$ and moreover there is a divisor-class $\cH$ on $\cX$  which restricts to $H'$ on $X_0$; we let $H_s:=\cH|_{X_s}$. We will replace $(X',H')$ by $(X_s,H_s)$ for $s$ very general in $S$ in order to ensure that $H^4(X_s)$ has the simplest possible Hodge structure. First we describe the Hodge substructures of  $H^4(X_s)$ that are forced by the Beauville-Bogomolov quadratic form and the integral $(1,1)$ class $\psi_t(\alpha)$. Let $X$ be a HK manifold. The Beaville-Bogomolov quadratic form $q_X$ provides us with a non-trivial class $q_X^{\vee}\in H^{2,2}_{\QQ}(X)$. In fact since $q_X$ is non-degenerate it defines an isomorphism
\begin{equation}
L_X\colon H^2(X)\overset{\sim}{\lra} H^2(X)^{\vee}.
\end{equation}
Viewing $q_X$ as a symmetric tensor in $H^2(X)^{\vee}\otimes H^2(X)^{\vee}$ and applying $L_X^{-1}$ we get a class $(L_X^{-1}\otimes L_X^{-1})(q_X)\in H^2(X)\otimes H^2(X)$; applying the cup-product map $H^2(X)\otimes H^2(X)\to H^4(X)$ to $(L_X^{-1}\otimes L_X^{-1})(q_X)$ we get an element $q_X^{\vee}\in H^{4}(X;\QQ)$ which is of type $(2,2)$ by Equation~\eqref{tipoqu}. Now we assume that $X$ is a numericla Hilbert square and that $H$ is a divisor class such that $q(H)=2$. Let $h:=c_1(\cO_X(H))$. We have an orthogonal (with respect to $q_X$) direct sum decomposition
\begin{equation}\label{h2decomp}
 H^2(X)=\CC h\wh{\oplus} h^{\bot}
\end{equation}
 into Hodge substructures of levels $0$ and $2$ respectively.
 Since $b_2(X)=23$ we get by~\Ref{crl}{ventitre}  that cup-product defines  an isomorphism
 \begin{equation}\label{symmeq}
Sym^2 H^2(X)\overset{\sim}{\lra} H^4(X).
\end{equation}
Because of~\eqref{symmeq} we will identify $H^4(X)$ with $Sym^2 H^2(X)$.
Thus~\eqref{h2decomp} gives a direct sum decomposition
\begin{equation}\label{symmdecomp}
    H^4(X)=\CC h^2\oplus \left(\CC h\otimes h^{\bot}\right)\oplus Sym^2(h^{\bot})
\end{equation}
into Hodge substructures of levels $0$, $2$ and $4$
 respectively. As is easily checked $q_X^{\vee}\in(\CC h^2\oplus Sym^2(h^{\bot})$.
 Let
\begin{equation}\label{wdef}
    W(h):=(q^{\vee})^{\bot}\cap Sym^2(h^{\bot}).
\end{equation}
(To avoid misunderstandings: the first orthogonality is with respect to the intersection form on
$H^4(X)$, the second one is with respect to
$q_X$.) One proves easily (see Claim~3.1 of~\cite{og3}) that $W(h)$ is a codimension-$1$ rational
sub Hodge structure of $Sym^2(h^{\bot})$, and that we have a direct sum
decomposition
\begin{equation}\label{smalldecomp}
\CC h^2\oplus Sym^2(h^{\bot})=\CC h^2\oplus \CC q^{\vee}\oplus W(h).
\end{equation}
Thus   we have the  decomposition
\begin{equation}\label{newsymmdecomp}
    H^4(X;\CC)=\left(\CC h^2\oplus \CC q^{\vee}\right)\oplus
    \left(\CC h\otimes h^{\bot}\right)\oplus W(h)
\end{equation}
into sub-H.S.'s of levels $0$, $2$ and $4$
respectively. The following result is  Proposition~3.2 of~\cite{og3}.
\begin{clm}\label{clm:hodgeprop}
Keep notation as above.  Let  $s\in S$ be very general i.e.~outside a countable union of proper analytic subsets of $S$. Then the following hold: 
 \begin{itemize}
\item[(1)]
  $H^{1,1}_{\ZZ}(X_s)=\ZZ h_s$ where $h_s=c_1(\cO_{X_s}(H_s))$.
\item[(2)]
Let $\Sigma\in Z_1(X_s)$ be an integral algebraic
$1$-cycle on $X_s$ and $cl(\Sigma)\in
H^{3,3}_{\QQ}(X_3)$ be its Poincar\'e dual. Then
$cl(\Sigma)=m h_s^3/6$ for some $m\in\ZZ$.
\item[(3)]
If $V\subset H^4(X_s)$ is a rational sub Hodge
structure then $V=V_1\oplus V_2\oplus V_3$ where
$V_1\subset\left(\CC h_s^2\oplus \CC q_{X_s}^{\vee}\right)$,
$V_2$ is either $0$ or equal to $\CC h_s\otimes
h_s^{\bot}$ and $V_3$ is either $0$ or equal to
$W(h_s)$.
\item[(4)]
The image of $h_s^2$ in $H^4(X_s;\ZZ)/Tors$ is indivisible.
\item[(5)]
$H^{2,2}_{\ZZ}(X_s)/Tors\subset\ZZ(h_s^2/2) \oplus\ZZ
(q_{X_s}^{\vee}/5)$.
\end{itemize}
\end{clm}
Let $s\in S$ be such that Items(1) through (5) of~\Ref{clm}{hodgeprop} hold. Let $X:=X_s$, $H:=H_s$ and $h:=c_1(\cO_X(H))$. Since $H$ is in the positive cone and $h$ generates $H^{1,1}_{\ZZ}(X)$ we get that $H$ is ample. By construction $X$ is a deformation of our given numerical Hilbert square. The goal is to analyze the linear system $|H|$. First we compute its dimension. A computation (see pp.~564-565 of~\cite{og3}) gives that $c_2(X)=6 q_X^{\vee}/5$; it follows that Equation~\eqref{rrhilbsq} holds for numerical Hilbert squares. Thus $\chi(\cO_X(H))=6$. By Kodaira vanishing we get that $h^0(\cO_X(H))=6$. Thus we have the map
\begin{equation}\label{mappeffe}
f\colon X\dashrightarrow |H|^{\vee}\cong\PP^5.
\end{equation}
The following is the main result of~\cite{og3}.
\begin{thm}{\rm [O'Grady]}\label{thm:mainthm1}
Let $(X,H)$ be as above. One of the following holds:
 \begin{itemize}
\item[(a)]
The line-bundle $\cO_X(H)$ is globally generated and there exist an anti-symplectic involution
 $\phi\colon X\to X$ and an inclusion
 $X/\la\phi\ra\hra |H|^{\vee}$  such that the map $f$ of~\eqref{mappeffe} is identified with the composition
 \begin{equation}
X\overset{\rho}{\lra}X/\la\phi\ra\hra |H|^{\vee}
\end{equation}
where $\rho$ is the quotient map. 
\item[(b)]
The map $f$ of~\eqref{mappeffe} is birational onto its image (a hypersurface of degree between $6$ and $12$). 
\end{itemize}
\end{thm}
\n
{\it Sketch  of proof.\/}
The following result follows from Items~(4) and~(5) of~\Ref{clm}{hodgeprop} plus a straightforward computation, see Proposition~4.1 of~\cite{og3}.
\begin{clm}\label{clm:intprop}
 If $D_1,D_2\in |H|$
are distinct then $D_1\cap D_2$ is a reduced
irreducible surface.
\end{clm}
In fact we chose $h$ such that $h^2$ is not divisible in $H^4(X;\ZZ)/Tors$ precisely to ensure that the above claim holds. Let $Y\subset\PP^5$ be the image of $f$ (to be precise the closure of the image by $f$ of its  regular points). Thus (abusing notation) we have $f\colon X\dashrightarrow Y$. Of course $\dim Y\le 4$. Suppose that $\dim Y=4$ and that $\deg f=2$. Then there exists a non-trivial rational involution $\phi\colon X\dashrightarrow X$ commuting with $f$. Since $\Pic (X)=\ZZ[H]$ we get that $\phi^{*}H\sim H$; since $K_X\sim 0$ it follows that $\phi$ is regular; it follows easily that~(a) holds. Thus it suffices to reach a contradiction assuming that $\dim Y<4$ or $\dim Y=4$ and $\deg f>2$. One goes through a (painful) case-by-case analysis. In each case, with the exception of $Y$ a quartic $4$-fold, one invokes either~\Ref{clm}{intprop} or Item~(3) of~\Ref{clm}{hodgeprop}. We give two \lq\lq baby\rq\rq cases.   First suppose that $Y$ is a quadric $4$-fold. Let $Y_0$ be an open dense subset containing the image by $f$ of its regular points. There exists a $3$-dimensional linear space $L\subset\PP^5$ such that $L\cap Y_0$ is a reducible surface. Now $L$ corresponds to the intersection of two distinct $D_1,D_2\in |H|$ and since $L\cap Y_0$ is  reducible so is $D_1\cap D_2$ - that contradicts~\Ref{clm}{intprop}. As second example we suppose that $Y$ is a smooth cubic $4$-fold and $f$ is regular. Notice that
\begin{equation}\label{dodici}
H\cdot H\cdot H\cdot H=12
\end{equation}
 by~\eqref{polarizzo} and hence $\deg f=4$.
Let $H^4(Y)_{pr}\subset H^4(Y)$ be the primitive cohomology. By Item~(3) of~\Ref{clm}{hodgeprop} we must have $f^{*} H^4(Y)_{pr}\subset \CC h\otimes h^{\bot}$. The restriction to  $f^{*} H^4(Y;\QQ)_{pr}$ of the intersection form on $H^4(X)$ equals  the intersection form on $H^4(Y;\QQ)_{pr}$ multiplied by $4$ because $\deg f=4$; one gets a contradiction by comparing discriminants. 
\qed
\begin{cnj}\label{cnj:adaesse}
Item~(b) of~\Ref{thm}{mainthm1} does not occur.
\end{cnj}
As we will explain in the next subsection~\Ref{cnj}{adaesse} implies that a numerical Hilbert square is in fact a deformation of $K3^{[2]}$. The following question arised in connection with the proof of~\Ref{thm}{mainthm1}.
\begin{qst}\label{qst:comek3}
Is the following true? Let $X$ be a HK $4$-fold and $H$ an ample divisor on $X$. Then $\cO_X(2H)$ is globally generated. 
\end{qst}
The analogous question in $dim=2$ has a positive answer, see for example~\cite{may}. We notice  that if $X$ is a  $4$-fold  with trivial canonical bundle and $H$ is ample on $X$ then $\cO_X(5H)$ is globally generated by Kawamata~\cite{kawa}. The relation between~\Ref{qst}{comek3} and~\Ref{thm}{mainthm1} is the following. 
\begin{clm}
Suppose that the answer to~\Ref{qst}{comek3} is positive. Let $X$ be a numerical Hilbert square equipped with an ample divisor $H$ such that $q_X(H)=2$.  Let $Y\subset  |H|^{\vee}$ be the closure of the image of the set of regular points of the rational map  $X\dashrightarrow |H|^{\vee}$. Then one of the following holds:
\begin{itemize}
\item[(1)]
$\cO_X(H)$ is globally generated.
\item[(2)]
$Y$ is contained in a quadric.
\end{itemize}
\end{clm}
\begin{proof}
 Suppose that Item~(2) does not hold. Then  multiplication of sections defines an injection $Sym^2 H^0(\cO_X(H))\hra H^0(\cO_X(2H))$; on the other hand we have
\begin{equation}
\dim Sym^2 H^0(\cO_X(H))=21=\dim H^0(\cO_X(2H)).
\end{equation}
(The last equation holds by Equation~\eqref{rrhilbsq} - valid for numerical Hilbert squares as noticed above.) Since $\cO_X(2H)$ is globally generated it follows that  $\cO_X(H)$ is globally generated as well i.e.~Item~(1) holds. 
\end{proof}
We remark that Items~(1) and~(2) of the above claim are not mutually exclusive. In fact let $S\subset\PP^3$ be a smooth quartic surface (a $K3$) not containing lines. We have a finite map
\begin{equation}\label{marrazzo}
\begin{matrix}
 S^{[2]} & \overset{f}{\lra} & \GG r(1,\PP^3)\subset\PP^5 \\
 [Z] & \mapsto & \la Z\ra
\end{matrix}
\end{equation}
with image the Pl\"ucker quadric in $\PP^5$. Let $H:=f^{*}\cO_{\PP^5}(1)$; since $f$ is finite $H$ is ample. Moreover $q(H)=2$ because $H\cdot H\cdot H\cdot H=12$;  thus~\eqref{marrazzo} may be identified with the map associated to the complete linear system $|H|$. 
 \subsection{Double EPW-sextics, II}\label{}
\setcounter{equation}{0}
 Let $(X,H)$ be as in Item~(a) of~\Ref{thm}{mainthm1}: we proved~\cite{og2} that there exists $A\in{\mathbb LG}(\wedge^3 \CC^6)^0$ such that $Y_A=f(X)$ and  the double cover $X\to f(X)$ may be identified with the canonical double cover $X_A\to Y_A$.  Since $X_A$ is a deformation of $K3^{[2]}$ it follows that if~\Ref{cnj}{adaesse} holds then numerical Hilbert squares are deformations of $K3^{[2]}$. The precise result proved in~\cite{og2} is the following.
\begin{thm}{\rm [O'Grady]}\label{thm:mainthm2}
Let $X$ be a  numerical Hilbert square. Suppose that $H$ is an ample divisor class on $X$ such that the following hold:
\begin{itemize}
\item[(1)]
$q_X(H)=2$ (and hence $\dim |H|=5$).
\item[(2)]
 $\cO_X(H)$ is globally generated. 
\item[(3)]
There exist an anti-symplectic involution
 $\phi\colon X\to X$ and an inclusion
 $X/\la\phi\ra\hra |H|^{\vee}$  such that the map $X\to |H|^{\vee}$   is identified with the composition
 \begin{equation}
X\overset{\rho}{\lra}X/\la\phi\ra\hra |H|^{\vee}
\end{equation}
where $\rho$ is the quotient map. 
\end{itemize}
Then there exists $A\in{\mathbb LG}(\wedge^3 \CC^6)^0$ such that $Y_A=Y$  and  the double cover $X\to f(X)$ may be identified with the canonical double cover $X_A\to Y_A$. 
\end{thm}
The proof of the above result goes as follows. 
\vskip 2mm
\n
{\it Step I.\/} Let $Y:=f(X)$; abusing notation we let $f\colon X\to Y$ be the double cover which is identified with the quotient map for the action of $\la\phi\ra$. We have the decomposition $f_{*}\cO_X=\cO_Y\oplus\eta$ where $\eta$ is the $(-1)$-eigensheaf for the action of $\phi$ on $\cO_X$. One proves that $\zeta:=\eta\otimes\cO_Y(3)$ is globally generated - an intermediate step is the proof that $3H$ is very ample. Thus we have an exact sequence
\begin{equation}\label{succdigi}
  0\to G\lra H^0(\zeta)\otimes\cO_{|H|^{\vee}}
  {\lra} i_{*}\zeta\to 0.
\end{equation}
where $i\colon Y\hra |H|^{\vee}$ is inclusion.
\vskip 2mm
\n
{\it Step II.\/} One computes $h^0(\zeta)$ as follows. First $H^0(\zeta)$ is equal to $H^0(\cO_X(3H))^{-}$ i.e.~the space of $\phi$-anti-invariant sections of $\cO_X(3H)$. Using Equation~\eqref{rrhilbsq} one gets that $h^0(\zeta)=10$. A local computation shows that $G$ is locally-free.  By invoking Beilinsons' spectral sequence for vector-bundles on projective spaces one gets that $G\cong \Omega^3_{|H|^{\vee}}(3)$. On the other hand one checks easily (Euler sequence) that the vector-bundle $F$ of~\eqref{eccoeffe} is isomorphic to $\Omega^3_{\PP(V)}(3)$. Hence if we identify $\PP(V)$ with $|H|^{\vee}$ then $F$ is isomorphic to the sheaf $G$ appearing in~\eqref{succdigi}. In other words~\eqref{succdigi} starts looking like the top horizontal sequence of~\eqref{spqr}.  
\vskip 2mm
\n
{\it Step III.\/} The multiplication map $\eta\otimes\eta\to\cO_Y$ defines an isomorphism $\beta\colon
  i_{*}\zeta\overset{\sim}{\lra}
  \text{\it Ext}^1
  (i_{*}\zeta,\cO_{|H|^{\vee}})$. Applying general results of
Eisenbud-Popescu-Walter~\cite{epw} (alternatively see the proof of Claim~(2.1) of~\cite{cascat}) one gets that $\beta$ fits into a 
  commutative diagram
\begin{equation}\label{grancomm}
 \begin{array}{ccccccccc}
0 & \mapor{} & \Omega^3_{|H|^{\vee}}(3)
&\mapor{\kappa}&
H^0(\theta)\otimes\cO_{|H|^{\vee}}  &
\mapor{} & i_{*}\zeta &
\to & 0\\
 & & \mapver{s^{t}}& &\mapver{s} &
&
\mapver{\beta}& & \\
0 & \mapor{} &
H^0(\theta)^{\vee}\otimes\cO_{|H|^{\vee}} &
\mapor{\kappa^{t}}&
\Theta^3_{|H|^{\vee}}(-3) & \mapor{\partial} &
\text{\it Ext}^1(i_{*}\zeta,\cO_{|H|^{\vee}}) & \mapor{}
& 0
\end{array}
\end{equation}
where the second row is obtained from the first one by applying $\text{\it Hom}(\cdot,\cO_{|H|^{\vee}})$. 
\vskip 2mm
\n
{\it Step IV.\/} One checks that
\begin{equation}
 \Omega^3_{|H|^{\vee}}(3)\mapor{(\kappa,s^{t})}
 \left(H^0(\zeta)\oplus
 H^0(\zeta)^{\vee}\right)
 \otimes\cO_{|H|^{\vee}}
\end{equation}
is an injection of vector-bundles. The transpose of the above map induces an isomorphism $\left(H^0(\zeta)^{\vee}\oplus
 H^0(\zeta)\right)\overset{\sim}{\lra} H^0( \Omega^3_{|H|^{\vee}}(3)^{\vee})$. The same argument shows that the transpose of~\eqref{eccoeffe}  induces an isomorphism $\wedge^{3}V^{\vee} \overset{\sim}{\lra} H^0(F^{\vee})$. Since $F$ is isomorphic to $\Omega^3_{|H|^{\vee}}(3)$ we get 
 an isomorphism
$  \rho\colon H^0(\zeta)\oplus
H^0(\zeta)^{\vee}\mapor{\sim} \wedge^3 V$
 such that (abusing notation) $\rho(\Omega^3_{|H|^{\vee}}(3))=F$. Lastly one checks that the standard symplectic form on $(H^0(\zeta)\oplus
H^0(\zeta)^{\vee})$ is identified (up to a multiple) via $\rho$ with the symplectic form $(,)_V$ of~\eqref{maitresse}. Now let $A=\rho(H^0(\zeta)^{\vee})$; then~\eqref{grancomm} is identified with~\eqref{spqr}. This ends the proof of~\Ref{thm}{mainthm2}.
 \section{Global Torelli and deformations of $K3^{[2]}$}\label{toratora}
 \setcounter{equation}{0}
The following question is motivated by the celebrated Global Torelli Theorem for $K3$ surfaces.
\begin{qst}\label{qst:rischiatutto}
 Let $\cC$ be a deformation class of HK manifolds. Is the following true? Let $X,Y$ be HK manifolds whose deformation class is $\cC$: then $X$ is bimeromorphic to $Y$ if and only if  there exists an integral Hodge isometry  $H^2(X)\cong H^2(Y)$.
\end{qst}
If the answer to the above question is affirmative we say that {\it Naive Global Torelli } holds for HK manifolds whose deformation class is $\cC$.  The reason we do not ask for a biregular Global Torelli  is that
bimeromorphic HK   manifolds have isomorphic $H^2$ cohomologies by Item~(5) of~\Ref{rmk}{quix} and 
 in  dimension greater than $2$ there do exist examples of  bimeromorphic HK   manifolds  which  
 are not isomorphic, see for example~\cite{deb} or the domain and codomain of the birational Map~\eqref{noniso}.  Notice also that  bimeromorphic HK manifolds are deformation equivalent  by  Huybrechts'~\Ref{thm}{birdef}. It is known that Naive Global Torelli does not hold for arbitrary deformation classes of HK manifolds.   
Namikawa~\cite{namik} proved that it is false for the deformation class of  $K^{[n]} T$ as soon as $n\ge 2$:  in fact  $K^{[n]} T$ and $K^{[n]} \wh{T}$ have isomorphic $H^2$'s but Namikawa proved that in general they are not bimeromorphic.
Markman~\cite{mark} proved that  if $(n-1)$ is not a prime power\footnote{Here we assume that $n>1$ of course.} then Naive Global Torelli fails for deformations of  $(K3)^{[n]}$. A refined Global Torelli  Question for deformations of $K3^{[n]}$ (based on work of Markman~\cite{mark}) has been formulated by Gritsenko, Hulek and Sankaran~\cite{ghs}; if $(n-1)$ is a prime power the refined and naive questions coincide. The recent preprint~\cite{verbtor} by Verbitsky presents a proof of a result which in particular gives an affirmative answer to~\Ref{qst}{rischiatutto} for deformations of $K3^{[n]}$ and $n-1$ a prime power.  Here we will not discuss Verbitsky's paper, instead we will concentrate on the deformation class of $K3^{[2]}$. In the first subsection we will assume that naive Global Torelli holds 
for deformations of $K3^{[2]}$ and we will derive a few easy (but interesting!) geometric consequences. In the second subsection we will give an outline of our work in progress on moduli and periods of double EPW-sextics. 
 \subsection{Torelli and geometry}\label{torhilbsq}
\setcounter{equation}{0}
In this subsection we make the following
\begin{ass}\label{ass:torass}
Naive Global Torelli holds for deformations of $K3^{[2]}$.
\end{ass}
 The first consequence that we will derive from the above assumption is about moduli spaces of polarized deformations of $K3^{[2]}$. (For a more general discussion see~\cite{ghs}.)
 First we recall a few results on lattices. Let $\Lambda$ be a lattice i.e.~a free finitely generated abelian group equipped with an integral bilinear symmetric form - we denote by $(,)_{\Lambda}$ the bilinear form  and by $q_{\Lambda}$ the associated quadratic form. We recall that
 \begin{equation}
H^2(K3^{[2]};\ZZ)\cong U^3\widehat{\oplus}E_8\la-1\ra^2
 \widehat{\oplus}\la -2\ra=:\Theta
\end{equation}
The {\it divisibility} of  $v\in \Theta$ is  
\begin{equation}\label{visibilita}
div(v):=|\ZZ/\{(v,w)\mid \ w\in \Theta\}|\,.
\end{equation}
Let $v$ be primitive: then $div(v)$ is either $1$ or $2$ and if it equals $2$ then $q_{\Theta}(v)\equiv 6\pmod{8}$. The following result is a corollary of Nikulin's general results on lattices~\cite{nik}.
\begin{clm}
 Let $v,w\in \Theta$ be primitive. There exists an isometry $\phi\in O(\Theta)$ such that $\phi(v)=w$ if and only if
 \begin{equation}
q_\Theta(v)=q_\Theta(w),\qquad div(v)=div(w)
\end{equation}
\end{clm}
 Let $d$ be a strictly positive integer and $\epsilon\in\{1,2\}$. We let $\gM_{2d}^{\epsilon}$ be the coarse moduli space for deformations $X$ of $K3^{[2]}$ equipped with a primitive ample divisor $H$ such that
 \begin{equation}
q_X(H)=2d,\qquad div(\cO_X(H))=\epsilon.
\end{equation}
(See~\cite{ghs} for details.) The period moduli space for such couples $(X,H)$ is defined as follows. Let $v\in \Theta$ be primitive such that $q_\Theta(v)=2d$ and $div(v)=\epsilon$. Let 
\begin{eqnarray}
\Omega_{v^{\bot}}:= & \{[\sigma]\in\PP(v^{\bot}\otimes_{\ZZ}\CC) \mid q_\Theta(\sigma)=0,\quad
 (\sigma,\ov{\sigma})_\Theta>0\},\\
 O(\Theta)_v:= & \{\phi\in O(\Theta)\mid \phi(v)=v\}. 
\end{eqnarray}
Then $O(\Theta)_v$ acts properly discontinuously on $\Omega_{v^{\bot}}$; thus the quotient $\DD_{2d}^{\epsilon}:=\Omega_{v^{\bot}}/O(\Theta)_v$ is an analytic space, in fact a quasi-projective variety by a classical result of Baily and Borel. One may define a period map
\begin{equation}
\gM_{2d}^{\epsilon}\overset{\gp_{2d}^{\epsilon}}{\lra} \DD_{2d}^{\epsilon}
\end{equation}
 proceeding as in the definition of~\eqref{succofrutta},  with the extra constraint that $\psi(v)=c_1(\cO_X(H))$. 
\begin{clm}\label{clm:sciarra}
If~\Ref{ass}{torass} holds then $\gp_{2d}^{\epsilon}$ is an isomorphism onto an open dense subset. In particular $\gM_{2d}^{\epsilon}$ is irreducible. 
\end{clm}
\begin{proof}
The period map $\gp_{2d}^{\epsilon}$ has finite fibers and it has open image because  the local period map is surjective and $\DD_{2d}^{\epsilon}$ is normal. Thus it suffices to prove the following:
\begin{itemize}
\item[(1)]
$\gM_{2d}^{\epsilon}$ is not empty.
\item[(2)]
$\deg \gp_{2d}^{\epsilon}=1$.
\end{itemize}
Let 
\begin{equation}
\DD_{2d}^{\epsilon}(1):=\{[\sigma]\in \DD_{2d}^{\epsilon}
\mid \sigma^{\bot}\cap \Theta=\ZZ v\}.
\end{equation}
(An element of $\DD_{2d}^{\epsilon}$ is a $ O(\Theta)_v$-orbit in $\Omega_{v^{\bot}}$;   to simplify notation we denote it by a representative $[\sigma]$.)
Since $\DD_{2d}^{\epsilon}(1)$ is dense in $\DD_{2d}^{\epsilon}$ it suffices to prove that
\begin{equation}
| (\gp_{2d}^{\epsilon})^{-1}([\sigma]) | = 1\qquad 
\forall [\sigma]\in \DD_{2d}^{\epsilon}(1).
\end{equation}
Let $[\sigma]\in \DD_{2d}^{\epsilon}(1)$. By~\Ref{thm}{tuttiper} there exist a deformation $X$ of $K3^{[2]}$ and a marking $\psi\colon \Theta\overset{\sim}{\lra} H^2(X;\ZZ)$ such that $\psi([\sigma])=H^{2,0}(X)$. Since $v\bot\sigma$ we have $\psi(v)\in H^{1,1}_{\ZZ}(X)$ and since $q_X(\psi(v))=q_\Theta(v)=2$ we get that $X$ is projective by Huybrechts' projectivity criterion~\ref{procri}. Moreover since $[\sigma]\in \DD_{2d}^{\epsilon}(1)$ we know that $\psi(v)$ generates $H^{1,1}_{\ZZ}(X)$ and hence $\pm\psi(v)$ is ample. Multiplying $\psi$ by $(-1)$ if necessary we may assume that $\psi(v)$ is ample. Let $H$ 
be a divisor class $H$ on $X$ such that $c_1(\cO_X(H))=\psi(v)$; then $\gp_{2d}^{\epsilon}(X,H)=[\sigma]$.  This proves that $(\gp_{2d}^{\epsilon})^{-1}([\sigma])$ is not empty. Let $[(X,H)],[(X',H')]\in \gM_{2d}^{\epsilon}$ be such that 
\begin{equation}
\gp_{2d}^{\epsilon}(X,H)=\gp_{2d}^{\epsilon}(X',H')\in 
\DD_{2d}^{\epsilon}(1).
\end{equation}
Let's prove that $[(X,H)]=[(X',H')]$. By~\Ref{ass}{torass} there exists a birational map $\phi\colon X\dashrightarrow X'$. Since $\gp_{2d}^{\epsilon}(X',H')\in \DD_{2d}^{\epsilon}(1)$ we have $H^{1,1}_{\ZZ}(X)=\ZZ c_1(\cO_X(H))$ and $H^{1,1}_{\ZZ}(X')=\ZZ c_1(\cO_{X'}(H'))$; it follows that $\phi^{*}H'\sim H$ and hence $\phi$ is a regular isomorphism because $H,H'$ are ample and $X,X'$ have trivial canonical bundle. 
\end{proof}
Next we show that~\Ref{ass}{torass} is linked to the conjectural converse of~\Ref{thm}{ampiose}.  
\begin{clm}
Suppose that~\Ref{ass}{torass} holds. 
Then Item~(2) of~\Ref{cnj}{diveff} holds.
\end{clm}
\begin{proof}
Let $X,L$ be as in  Item~(2) of~\Ref{cnj}{diveff}. Suppose first that
\begin{equation}\label{genpic}
H^{1,1}_{\ZZ}(X)=\ZZ c_1(L).
\end{equation}
The subspace   $c_1(L)^{\bot}\subset H^2(X)$ is a sub Hodge structure because  $c_1(L)\in H^{1,1}_{\ZZ}(X)$.   
Let $c_1(L)^{\bot}_{\ZZ}:= c_1(L)^{\bot}\cap H^2(X;\ZZ)$; then
\begin{equation}\label{brendona}
H^2(X;\ZZ)=c_1(L)^{\bot}_{\ZZ}\oplus \ZZ c_1(L)
\end{equation}
because $q_X(L)=-2$ and $(c_1(L),H^2(X;\ZZ))_X=2\ZZ$. By~\eqref{brendona} the lattice $c_1(L)^{\bot}_{\ZZ}$ is even, unimodular of signature $(3,19)$; it follows that it is isometric to $U^3\widehat{\oplus}E_8\la-1\ra^2$ i.e.~the $K3$ lattice. By surjectivity of the period map for $K3$ surfaces (i.e.~\Ref{thm}{tuttiper}) there exist a $K3$ surface $S$ and an isomorphism of integral Hodge structures $\phi_0\colon H^2(S)\overset{\sim}{\lra} c_1(L)^{\bot}$ which is an isometry. By~\eqref{comhilb} and~\eqref{rethilb} $\phi_0$ extends to an isomorphism of integral Hodge structures 
\begin{equation}
\phi\colon H^2(S^{[2]})\overset{\sim}{\lra} H^2(X)
\end{equation}
 which is an isometry.  By~\Ref{ass}{torass} there exists a bimeromorphic map $f\colon X\dashrightarrow S^{[2]}$. Since $H^{1,1}_\ZZ(X)=\ZZ c_1(L)$ it follows that $f^{*}\xi_2=\pm c_1(L)$. Let $\Delta_2\subset S^{[2]}$ be the effective divisor parametrizing non-reduced analytic subsets; then $f^{-1}_{*}\Delta_2$ is an effective divisor and by~\eqref{doppio} we have  $c_1(\cO_X(f^{-1}_{*}\Delta_2))=c_1(L^{\pm 2})$. This proves that Item~(2) of~\Ref{cnj}{diveff} holds   if we make the extra assumption~\eqref{genpic}. In general  one may proceed as follows. Let $\pi\colon\cX\to T$ be a representative of $Def(X,L)$. We let $X_t:=\pi^{-1}(t)$ and $0\in T$ such that $X_0\cong X$. Of course we have a line-bundle $\cL$ on $\cX$ such that $\cL|_{X_0}\cong L$; we let $L_t:=\cL|_{X_t}$. The set 
 \begin{equation}
T_{gen}:=\{t\in T\mid H^{1,1}_\ZZ(X_t)=\ZZ c_1(\cL_t)\}
\end{equation}
is dense in $T$. By what we have proved one gets that either $h^0(L_t^2)>0$ for all $t\in T_{gen}$ or else $h^0(L_t^{-2})>0$ for all $t\in T_{gen}$. We may assume that the former holds; by upper semi-continuity of cohomology dimension it follows that 
 $h^0(L_t^2)>0$ for all $t\in T$, in particular $h^0(L_0^2)>0$.
\end{proof} 
\begin{rmk}
It should be possible to prove that~\Ref{ass}{torass} implies that Item~(2) of~\Ref{cnj}{diveff} holds and moreover that  the conjectural converse of~\Ref{thm}{ampiose} is true. The proof will be somewhat less elementary because the generic deformation of $K3^{[2]}$ satsifying the hypothesis of Item~(2) of~\Ref{cnj}{diveff} or Item~(b) of~\Ref{thm}{ampiose} is not bimeromorphic to a $K3^{[2]}$. The natural idea is to start from one $4$-fold satisfying the hypothesis and the conclusion and argue by  a  deformation argument and stability of the divisor in the first case and the lagrangian surface in the second case.
\end{rmk}
 \subsection{Double EPW-sextics and Torelli}
\setcounter{equation}{0}
 Let $V$ be a complex vector space of dimension $6$. The action of $PGL(V)$ on $\lagr$ lifts (uniquely) to an action  on the Pl\"ucker line-bundle i.e.~it is    linearized. Thus there is a GIT quotient 
\begin{equation}
 \mathfrak{M}:=\lagr// PGL(V)\,. 
\end{equation}
Given a semistable $A\in\lagr$ we let $[A]\in\gM$ be the corresponding point. Let  $A\in\lagr$ and assume that $Y_A\not=\PP(V)$; we let $H_A\in |f_A^{*}\cO_{\PP(V)}(1)$, thus $(X_A,H_A)$ is a polarized $4$-dimensional scheme, if $X_A$ is smooth i.e.~$A\in\lagr^0$ then it is a HK deformation of $K3^{[2]}$ of degree $2$. We note that if $A$ is semistable then $Y_A\not=\PP(V)$, that is proved in~\cite{ogemma}.
The open dense $\lagr^0\subset\lagr$ (see~\eqref{eccozero}) is contained in the stable locus of $\lagr$ (this follows easily from Proposition~6.1 of~\cite{og2}). Let
\begin{equation}
 \mathfrak{M}^0:=\lagr^0// PGL(V).
\end{equation}
One proves that points of $\gM^0$ are in one-to-one correspondence with 
isomorphism classes of double  EPW-sextics i.e.~$[A]=[B]$ if and only if the polarized HK $4$-folds $(X_A,H_A)$ and $(X_B,H_B)$ are isomorphic.  Let $\DD:=\DD^1_2$, notation as in Subsection~\ref{torhilbsq}. Let $A\in\lagr^0$; then $q(H_A)=2$ and hence the period map for double EPW-sextics is a regular map of quasi-projective varieties
$\gp^0\colon \gM^0\to\DD$. Let $\DD^{BB}\supset\DD$ be the Baily-Borel compactification; then $\gp^0$ extends to a rational map
\begin{equation}
\gp\colon \gM\dashrightarrow\DD^{BB}.
\end{equation}
By~\Ref{thm}{epwdoppie} we know that locally in the classic topology $\lagr^0$  parametrizes a locally versal family of HK deformations of $K3^{[2]}$ of degree $2$; it follows that $\gp$ is dominant of finite degree. The following claim gives one motivation for studying the period map $\gp$.
\begin{clm}\label{clm:natali}
Suppose that
\begin{itemize}
\item[(a)]
\Ref{cnj}{adaesse} holds and
\item[(b)]
$\deg\gp=1$.
\end{itemize}
Then Naive Global Torelli holds for deformations of $K3^{[2]}$. 
\end{clm}
Before proving the claim we discuss a few density results. Let $X$ be a deformation of $K3^{[2]}$. Let $\pi\colon\cX\to T$ be a representative of $Def(X)$. As usual we let $X_t:=\pi^{-1}(t)$ and $X_0\cong X$. We will assume that $T$ is small enough; that means that $T$ is simply connected and that the local period map~\eqref{mappaperiodi}  is an isomorphism onto an open (classical topology) subset of $V(q_X)\subset\PP(H^2(X))$. Since $T$ is simply connected the Gauss-Manin connection gives an identification
\begin{equation}\label{parallelo}
H^2(X)\overset{\sim}{\lra} H^2(X_t)\qquad \forall t\in T.
\end{equation}
Given $d\in\ZZ$ we let $T_{2d}\subset T$ be defined by
\begin{equation}
T_{2d}:=\{t\in T\mid H^{1,1}_{\ZZ}(X_t)\ni \gamma,\quad q(\gamma)=2d,\quad \text{$\gamma$ primitive}\}.
\end{equation}
The following result is proved by copying the proof of Proposition 2 of Le Potier's paper~\cite{lepotier}.
\begin{prp}\label{prp:joseph}
Keep notation as above. Then $T_{2d}$ is dense (classical topology) in $T$.
\end{prp}
We are interested in the case $d=1$ and we wish to show that a certain subset of $T_2$ is dense in $T$ as well. First we define complex multiplication HK manifolds. Let $X$ be a HK manifold such that the restriction of $q_X$ to $H^{1,1}_{\ZZ}(X)$ is non-degenerate, for example a projective one; the {\it transcendental lattice of $X$} is the sublattice $T(X)\subset H^2(X;\ZZ)$ perpendicular to $H^{1,1}_{\ZZ}(X)$. Let $T(X)_{\CC}:=T(X)\otimes_{\ZZ}\CC$; then $T(X)_{\CC}$ is a Hodge substructure of $H^2(X)$ and it is simple\footnote{It contains no non-trivial sub-H.S.} because $q_X$  is non-degenerate on $H^{1,1}_{\ZZ}(X)$. We say that {\it $X$ has complex multiplication (CM)} if there exists an endomorphism of the Hodge structure $T(X)_{\CC}$ which is not a homothety. Now let 
\begin{equation}
\cV:=\{\alpha\in H^2(X;\ZZ)\mid q_X(\alpha)=2\}. 
\end{equation}
Given $\alpha\in \cV$ we let 
\begin{equation}
T_{\alpha}:=\{t\in T\mid \alpha\in H^{1,1}_{\ZZ}(X_t)\}.
\end{equation}
The above definition makes sense because Gauss-Manin gives Identification~\eqref{parallelo}. Let $t\in T_{\alpha}$; since $q_{X_t}(\alpha)>0$ either $(\alpha,\cdot)_{X_t}$ is strictly positive or strictly negative on $\cC_{X_t}$, moreover the sign is independent of $t$ by continuity. Thus we have a disjoint union $\cV=\cV^{+}\coprod\cV^{-}$ where  
\begin{eqnarray}
\cV^{+}:= & \{\alpha\in\cV\mid (\alpha,\ \beta)_{X_t}>0\ \ \forall t\in T_{\alpha}
\ \ \forall\beta\in\cC_{X_t}\}\label{burgnich}\\
\cV^{-}:= & \{\alpha\in\cV\mid (\alpha,\ \beta)_{X_t}<0\ \ \forall t\in T_{\alpha}
\ \ \forall\beta\in\cC_{X_t}\}.\label{domenghini}
\end{eqnarray}
\begin{dfn}
Let $\alpha\in\cV^{+}$. We let $T_{\alpha}^{gen}\subset T_{\alpha}$ be the set of $t$ such that $X_t$ is \ul{not} CM (this makes sense: $q_{X_t}$ is non degenerate on $H^{1,1}_{\ZZ}(X_t)$ because $q_{X_t}(\alpha)>0$) and moreover Items~(1) through~(5) of~\Ref{clm}{hodgeprop} hold with $s=t$ and $h_s=\alpha$. 
\end{dfn}
\begin{prp}\label{prp:moltogen}
Keep notation as above. Then $T_{\alpha}^{gen}$ is dense (classical topology) in $T_{\alpha}$. 
\end{prp}
\begin{proof}
Let $T_{\alpha}(1)\subset T_{\alpha}$ be the set of $t$ such that $H^{1,1}_{\ZZ}(X_t)=\ZZ\alpha$. A standard argument gives that $T_{\alpha}(1)$ is the complement of a countable union of proper analytic subsets of $T_{\alpha}$. Let $T_{\alpha}(2)\subset T_{\alpha}(1)$ be the set of $t$ such that Item~(3) of~\Ref{clm}{hodgeprop} holds with $X_s=X_t$ and $h_s=\alpha$. Again by a standard argument  $T_{\alpha}(2)$ is the complement of a countable union of proper analytic subsets of $T_{\alpha}$ - see Lemma~3.3 of~\cite{og3}. Let $t\in T_{\alpha}(2)$; then  Items~(1) through~(5) of~\Ref{clm}{hodgeprop} hold with $s=t$ and $h_s=\alpha$. In fact~(1) and~(3) hold by definition, (4) holds by~\Ref{rmk}{casodefo}; Item~(2) follows from Item~(1) and Item~(5) follows from (3) and~(4) - see the proof of Proposition~3.2 of~\cite{og3}. Let $T_{\alpha}^{CM}\subset T_{\alpha}$   be the set of $t$ such that $X_t$ has complex multiplication and $T_{\alpha}^{CM}(1):=T_{\alpha}^{CM}\cap T_{\alpha}(1)$. We claim that 
\begin{equation}\label{loreti}
\text{$T_{\alpha}^{CM}$  is contained in a countable union of proper analytic subsets of $T_{\alpha}$.}
\end{equation}
 Since 
$(T_{\alpha}\setminus T_{\alpha}(1))$ is a countable union of proper analytic subsets of $T_{\alpha}$ it suffices to prove that $T_{\alpha}^{CM}(1)$ is contained in a countable union of proper analytic subsets of $T_{\alpha}$. Let $t\in T_{\alpha}^{CM}(1)$. Then $H^{1,1}(X_t)=\alpha^{\bot}$ and hence there exists an \ul{integral} homomorphism of groups $\phi\colon \alpha^{\bot}\to \alpha^{\bot}$ which is not a homothety and such that $H^{2,0}(X_t)$ is an eigenspace of $\phi$, say with eigenvalue $\lambda$. Since $\phi$ is not a homothety the $\lambda$-eigenspace $V_{\lambda}\subset \alpha^{\bot}$ is not all of $\alpha^{\bot}$; it follows that
\begin{equation}
\{t\in T_{\alpha}\mid H^{2,0}(X_t)\subset V_{\lambda}\}
\end{equation}
is a proper analytic subset of $T_{\alpha}$. The set of integral $\phi$ as above is countable; it follows that $T_{\alpha}^{CM}(1)$ is contained in a countable union of proper analytic subsets of $T_{\alpha}$; this proves~\eqref{loreti}. Since
$T_{\alpha}^{gen}=T_{\alpha}(2)\setminus T_{\alpha}^{CM}$ we get that the complement of $T_{\alpha}^{gen}$ in $T_{\alpha}$ is contained  in a countable union of proper analytic subsets of $T_{\alpha}$, in particular $T_{\alpha}^{gen}$ is dense in $T_{\alpha}$.
\end{proof}
\begin{crl}\label{crl:mazzola}
Keep notation as above. Then 
\begin{equation}\label{tuttiquanti}
\bigcup_{\alpha\in\cV^{+}} T_{\alpha}^{gen}
\end{equation}
 is dense in $T$.
\end{crl}
\begin{proof}
We have $T_2=\cup_{\alpha\in\cV^{+}} T_{\alpha}$. By~\Ref{prp}{moltogen} we get that the closure of~\eqref{tuttiquanti} equals the closure of $T_2$. Thus the Corollary follows from~\Ref{prp}{joseph}. 
\end{proof}
\n
{\it Proof of~\Ref{clm}{natali}.\/}
Suppose that $Z,Z'$ are HK deformations of $K3^{[2]}$ and that there exists an integral isomorphism of Hodge structures
\begin{equation}\label{memehodge}
\mu\colon H^2(Z)\overset{\sim}{\lra} H^2(Z')
\end{equation}
which is  an isometry with respect to the B-B forms. Composing $\mu$ with $-Id_{H^2(Z')}$ we may assume that
\begin{equation}\label{jair}
\mu(\cC_{Z})=\cC_{Z'}.
\end{equation}
By the existence of $\mu$ we may choose markings $\psi,\psi'$ of $Z$ and $Z'$ repectively such that $\cP(Z,\psi)=\cP(Z',\psi')$. Let $\pi\colon\cZ\to T$ and $\pi'\colon\cZ'\to T'$ be representatives of $Def(Z)$ and $Def(Z')$ with $T,T'$ small. (As usual $Z_t=\pi^{-1}(t)$, $Z_0\cong Z$ and $Z'_t=(\pi')^{-1}(t)$, $Z'_{0}\cong Z'$.) By infinitesimal Torelli and local surjectivity of the period map we may shrink $T$ and $T'$ so that there exists an isomorphism $g\colon T\to T'$ such that 
\begin{equation}\label{identici}
\cP(Z_t,\psi)=\cP(Z'_{g(t)},\psi').
\end{equation}
(Here $\psi$ defines a marking of $Z_t$ by Gauss-Manin and similarly for $\psi'$.) Let $\cV_Z^{+}\subset H^2(Z;\ZZ)$ and $\cV_{Z'}^{+}\subset H^2(Z';\ZZ)$ be defined as in~\eqref{burgnich}. By~\eqref{jair} we have $\mu(\cV_Z^{+})=\cV_{Z'}^{+}$. Let $\alpha\in  \cV_Z^{+}$; by~\eqref{identici} we have $g(T_{\alpha}^{gen})=T^{gen}_{\mu(\alpha)}$. Let $t\in T_{\alpha}^{gen}$. By~\Ref{cnj}{adaesse} and~\Ref{thm}{mainthm2}   there exist $A,A'\in\lagr^0$ such that $Z_t$, $Z'_{g(t)}$ are isomorphic to the double EPW-sextics $X_A$, $X_{A'}$ respectively. Moreover $\alpha$ and $\alpha'$ are the natural ample classes on $X_A$ and $X_{A'}$ respectively because they belong to $\cV_Z^{+}$ and $\cV_{Z'}^{+}$ respectively. We have 
$\gp([A])=\gp([A'])$ by~\eqref{identici};  since we are assuming that $\deg\gp=1$ it follows that $X_A\cong X_{A'}$. Let $f\colon X_{A'}\overset{\sim}{\lra} X_A$ be an isomorphism. The integral isomorphism of Hodge structures $H^2(f)\colon H^2(X_A)\overset{\sim}{\lra}  H^2(X_{A'})$ is an isometry with respect to the B-B forms; it sends $\alpha$ to $\alpha'$ and hence $\alpha^{\bot}$ to $(\alpha')^{\bot}$. Since $X_A,X_{A'}$ do \ul{not} have complex multiplication the restriction of $H^2(f)$ to  $\alpha^{\bot}$ is either equal to the restriction of $\mu$ or to the restriction of $-\mu$. If the latter occurs we replace $f$ with its composition with the covering involution of $X_A\to Y_A$ and we get that we may assume that $H^2(f)=\mu$.   
By~\Ref{crl}{mazzola} there exists a sequence $\{t_i\}$ converging to $0$ with  $t_i\in \cup_{\alpha\in\cV^{+}} T_{\alpha}^{gen}$ for all $i$.  For each $t_i$ we have an isomorphism $f_i\colon Z'_{g(t_i)}\overset{\sim}{\lra} Z_{t_i}$ such that $H^2(f_i)=\mu$; under these hypotheses Huybrechts (Theorem~4.3 of~\cite{huy}) proved that the \lq\lq Main Lemma\rq\rq of Burns-Rapoport~\cite{bura} extends to higher-dimensional HK's i.e.~a subsequence of the graphs of the $f_i$ converges to the graph of a  bimeromorphic map $Z'\dashrightarrow Z$.  
\qed
 \subsection{Periods of double EPW-sextics}
\setcounter{equation}{0}
The period maps for double EPW-sextics and for cubic hypesurfaces in $\PP^5$ have many common features. We will state the main results on periods of double EPW-sextics and then we will point out the analogies with the case of cubic $4$-folds. Let $\Delta,\Sigma\subset\lagr$ be defined by~\eqref{eccodel}, \eqref{eccosig}. One shows that $(\Delta\setminus\Sigma)$ is contained in the stable locus and that  the generic point of $\Sigma$ is stable; it follows that
\begin{equation}
\gT:=\Delta//PGL(V),\qquad \gN:=\Sigma//PGL(V)
\end{equation}
are prime divisors in $\gM$. The period map $\gp$ is not regular; one of the main issues is to determine the locus of regular points of $\gp$.  
One first proves that $\gp$ is regular away from $\gN$. In order to analyze $\gp$ at a point $x\in\gN$ we assume that $A$ belongs to the unique closed orbit\footnote{That is closed in the semistable locus $\lagr^{ss}$.} representing $x$. Suppose that $W\in \GG r(3,V)$ and that $\wedge^3 W\subset A$. One defines a subscheme $C_{W,A}\subset \PP(W)$ as in Item~(1) of~\Ref{rmk}{fibcon}; it is either a sextic curve or all of $\PP(W)$ (pathological case). We let $\gM^{\flat}\subset\gM$ be the locus of $[A]$  such that the following holds:  for all $W\in \GG r(3,V)$ such that $\wedge^3 W\subset A$ the scheme $C_{W,A}$ is a $PGL(W)$-semistable sextic which does not contain a triple conic in the closure of its orbit\footnote{The GIT-quotient of the space of plane sextics is a compactification of the moduli space of degree-$2$ $K3$ surfaces; the point corresponding to triple conics is the indeterminacy locus of the period map to the Baily-Borel compactification of the relevant period moduli space.}. The map $\gp$ extends regularly over $\gM^{\flat}$:
\begin{equation}
\begin{matrix}
\gM^{\flat} & \lra & \DD^{BB} \\
[A] & \mapsto & \gp([A]) 
\end{matrix}
\end{equation}
(In fact we guess that $\gM^{\flat}$ is equal to the set of regular points of $\gp$.)
Let $\gM^{ADE}\subset\gM^{\flat}$ be the locus of $[A]$  such that the following holds:  for all (or equivalently one) $W\in \GG r(3,V)$ such that $\wedge^3 W\subset A$ the scheme $C_{W,A}$  is a reduced sextic with ADE singularities i.e.~the double cover $S\to\PP(W)$ ramified over $C_{W,A}$ has at most DuVal singularities. One has $\gM^{ADE}=\gp^{-1}\DD\cap \gM^{\flat}$ and hence we have
\begin{equation}\label{platonico}
\begin{matrix}
\gM^{ADE} & \lra & \DD \\
[A] & \mapsto & \gp([A]) 
\end{matrix}
\end{equation}
Moreover Map~\eqref{platonico} has finite fibers. Next we analyze the restriction of $\gp$ to $\gT^{ADE}:=\gT\cap\gM^{ADE}$ and to $\gN^{ADE}:=\gN\cap\gM^{ADE}$. 
The double EPW-sextic parametrized by  $[A]\in (\gT\setminus\gN)$ is birational to $S_A(v_i)^{[2]}$ where $S_A(v_i)$ is the $K3$ surface described in the proof of~\Ref{thm}{epwdoppie}. Similarly let $[A]\in\gN$ be generic as in~\Ref{rmk}{fibcon}; then  the double cover $S_{W,A}\to\PP(W)$ ramified over   the smooth  sextic $C_{W,A}$  is a $K3$ surface. It is not the case that $X_A$ is birational to a Hilbert square but  $ e_A^{\bot}\subset H^2(X_A)$ is a sub-Hodge structure of $H^2(S_{W,A})$ of index  $2$. (Here $e_A$ is as in Item~(5) of~\Ref{rmk}{fibcon}.)  In both cases Global Torelli for $K3$'s and Riemann-Roch for $K3$ surfaces allow us to analyze the restriction of $\gp$ to $\gT^{ADE}$ and to $\gN^{ADE}$. The closure of $\gp(\gN^{ADE})$ in $\DD$  is an irreducible component ${\mathbb S}_2^{\star}$ of the 
divisor
\begin{equation}
\{[\sigma]\in\DD\mid \text{$\exists \gamma\in \{v,\sigma\}^{\bot}\cap\Theta$ such that $q_{\Theta}(\gamma)=-2$}\}.
\end{equation}
(Here $v\in\Theta$ is a fixed vector such that $q_{\Theta}(v)=2$ - see Subsection~\ref{torhilbsq}.)
Moreover the following hold:
\begin{itemize}
\item[(a)]
The restriction of $\gp$ to $\gN^{ADE}$ is injective.
\item[(b)]
$\gp$ is not ramified along $\gN$.
\item[(c)]
$\gp^{-1}({\mathbb S}_2^{\star})\cap\gM^{\flat}=\gN^{ADE}$.
\end{itemize}
Similar results hold for the period map on $\gT$. Let's pretend for a moment that $\gp$ is regular; then Items~(a)-(c) give  that $\deg\gp=1$ because $(\gM\setminus\gM^{\flat})$ contains no divisor - in fact it has relatively high codimension. Going back to the \lq\lq real\rq\rq world (i.e.~$\gp$ is not regular): if the dimension of $(\gM\setminus\gM^{\flat})$ is at most $6$ then one may adapt an argument of Voisin~\cite{claire} (see the Erratum) and derive $\deg\gp=1$ from~(a) through~(c) above. We do not yet know whether the required upper bound holds - what is missing is a complete (or detailed enough) analysis of GIT (semi)stability for the $PGL(V)$ action on $\lagr$. Now we go over the analysis of the period map for cubic hypersurfaces in $\PP^5$ according to Voisin~\cite{claire} and Laza~\cite{laza1,laza2} (see also~\cite{looicomp}). Let $|\cO_{\PP^5}(3)|^{spl}\subset 
|\cO_{\PP^5}(3)|$ be the open set parametrizing cubics with simple singularities. Then $|\cO_{\PP^5}(3)|^{spl}$ is $PGL(6)$-invariant  and by Laza~\cite{laza1} it is contained in the stable locus of $|\cO_{\PP^5}(3)|$. Let
\begin{equation}
\cM^{spl}:=|\cO_{\PP^5}(3)|^{spl}//PGL(6),\qquad
\cM_{cbc}:=|\cO_{\PP^5}(3)|//PGL(6).
\end{equation}
We have the  (rational)  period map is $\gp\colon \cM_{cbc}\dashrightarrow (\DD_6^2)^{BB}$ where $(\DD_6^2)^{BB}$ is the Baily-Borel compactification of the period moduli space described in Subsection~\ref{torhilbsq}. Then $\cM^{spl}$ is the analogue of the open $\cM^{ADE}$ in the moduli space of double EPW-sextics. In fact $\cM^{spl}=\gp^{-1}(\DD_6^2)\cap Reg(\gp)$ and moreover the restriction of $\gp$ to $\cM^{spl}$ has   finite fibers (of cardinality $1$ by Voisin's Global Torelli for cubics). Next let $\cD,\cP$ be the prime divisors of $|\cO_{\PP^5}(3)|$ defined in Subsection~\ref{bedofam}: as shown in that subsection the varieties of lines on cubics parametrized by points of $\cD$ are similar to double EPW-sextics parametrized by points of $\Delta$ and there is also an analogy between $\cP$ and $\Sigma$. Voisin~\cite{claire} proved that analogues of  Items~(a)-(c) above hold for $\cP// PGL(6)$ and from that derived Global Torelli for cubics. 
\end{document}